\documentclass[11pt]{amsart}
 \usepackage{amsmath,amssymb,enumerate}
 \usepackage[all]{xy}
 \usepackage{xcolor}
 \usepackage[T1]{fontenc}
\usepackage{upquote}

\usepackage{xcolor}

 \newtheorem{pro}{Proposition}[section]
 
 \newtheorem{lem}[pro]{Lemma}
 \newtheorem{Lem}[pro]{Lemma}
 \newtheorem{rem}[pro]{Remark}
 \newtheorem{exa}[pro]{Example}
 \newtheorem{Exa}[pro]{Example}
 
  \newtheorem{defi}[pro]{Definition}
 
 \newtheorem{thm}[pro]{Theorem}
 
 \newtheorem{ques}[pro]{Question}

 \theoremstyle{plain}
\newtheorem*{namedthm}{\namedthmname}
\newcounter{namedthm}

\makeatletter
\newenvironment{named}[1]
  {\def\namedthmname{#1}%
   \refstepcounter{namedthm}%
   \namedthm\def\@currentlabel{#1}}
  {\endnamedthm}

 \newcommand{\R}{\mathbb R}
 
 \newcommand{\C}{\mathbb C}
 \newcommand{\N}{\mathbb N}

 \newcommand{\e}{\varepsilon}
 \newcommand{\vep}{\varepsilon}
 \newcommand{\capo}{{\rm Cap}_{\omega}}
 \newcommand{\capfp}{{\rm Cap}_{\varphi,\psi}}
 \newcommand{\captheta}{{\rm Cap}_{\theta}}

 \newcommand{\f}{\varphi}
 
 \newcommand{\p}{\psi}

 \newcommand \psh {{\rm PSH}}

 \newcommand \MA {{\rm MA}}
 \newcommand \Amp {{\rm Amp}}
 \newcommand \vol{{\rm Vol}}
 
 \newcommand \Ec{\mathcal E }

 \newcommand \mrm{\mathrm}

 \numberwithin{equation}{section}

\usepackage{hyperref}
\hypersetup{
    bookmarks=true,         
    unicode=false,          
    pdftoolbar=true,        
    pdfmenubar=true,        
    pdffitwindow=false,     
    pdfstartview={FitH},    
    pdftitle={},    
    pdfauthor={},     
    colorlinks=true,       
   linkcolor=blue,          
    citecolor=blue,        
    filecolor=black,      
    urlcolor=blue}           

\begin{document}

 \title{Plurisubharmonic envelopes and supersolutions}
\author{Vincent Guedj, Chinh H. Lu, Ahmed Zeriahi}

\address{Institut de Math\'ematiques de Toulouse   \\ Universit\'e de Toulouse \\
118 route de Narbonne \\
31400 Toulouse, France\\}

\email{vincent.guedj@math.univ-toulouse.fr}

\address{Laboratoire de Math\'ematiques d'Orsay\\
 Univ. Paris-Sud\\
 CNRS, Universit\'e Paris-Saclay\\
  91405 Orsay, France.}

\email{hoang-chinh.lu@math.u-psud.fr}

\address{Institut de Math\'ematiques de Toulouse   \\ Universit\'e de Toulouse \\
118 route de Narbonne \\
31400 Toulouse, France\\}

\email{ahmed.zeriahi@math.univ-toulouse.fr}

\thanks{The authors are partially supported by the ANR project GRACK}

 \date{\today}
 
 \begin{abstract}
 We make a systematic study of (quasi-)plurisubharmonic envelopes on compact K\"ahler manifolds, as well as on domains of $\C^n$,
 by using and extending an approximation process due to Berman \cite{Ber13}. We show that the 
quasi-psh envelope of a viscosity super-solution  is a pluripotential super-solution
of a given complex Monge-Amp\`ere equation. We use these ideas to solve complex Monge-Amp\`ere equations by taking lower envelopes of super-solutions.
\end{abstract}

\maketitle

\tableofcontents

\section*{Introduction}
Weak subsolutions and supersolutions to complex Monge-Amp\`ere equations play a central role in the development of complex analysis and geometry.  
These   have been studied extensively, in the pluripotential sense, since the fundamental works of Bedford and Taylor \cite{BT76,BT82}.  

The notions of viscosity sub/super/solutions to  complex Monge-Amp\`ere equations  
\begin{equation}
\tag{CMA}
(\theta+dd^c u)^n =e^u f dV
\end{equation}
have been introduced in \cite{HL09,EGZ11,Wang12,HL13}. 
It has been notably shown in \cite[Theorem 1.9 and Proposition 1.11]{EGZ11} that an u.s.c. function is a viscosity subsolution if and only if it is
plurisubharmonic and a pluripotential subsolution.
The connection between viscosity and pluripotential supersolutions has however remained mysterious so far.
Our first main result gives a satisfactory answer to 
\cite[Question 40]{DGZ16}:

\begin{named}{Theorem A}\label{thmA}
	Let $v$ be a viscosity supersolution to a complex Monge-Amp\`ere equation $(CMA)$.
The (quasi-)plurisubharmonic envelope $P(v)$ is a pluripotential supersolution
to $(CMA)$.
\end{named}

The envelope $P(v)$ depends on the context (local/global) (see Section \ref{sec:envelope}).
We refer the reader to section \ref{sec:viscosity} for the definition of viscosity supersolutions
and to Theorems \ref{thm: visco super vs pluri compact} and \ref{thm:supersolloc} for more precise statements.

The proof of this result relies on an approximation scheme which is of inependent
interest. This method has been introduced by Berman (see \cite[Theorem 1.1]{Ber13})
when $v$ is smooth. We need to extend it here in order to deal with functions
$v$ which are  less regular, proving in particular the following :

\begin{named}{Theorem B}\label{thmB}
	Let $(X,\omega)$ be a $n$-dimensional compact K\"ahler manifold
and fix $\mu$  an arbitrary non-pluripolar positive measure.
 Let $v:X \rightarrow \R$ be a bounded Borel-measurable function.
 Let $\f_j \in \Ec^1(X,\omega)$ be the unique solution to the complex Monge-Amp\`ere equation
 $$
 (\omega+dd^c \f_j)^n =e^{j(\f_j-v)} \mu.
 $$
 
 Then $(\f_j)$ converges in capacity to  the $(\omega,\mu)$-envelope $P_{\omega,\mu}(v)$, where
 $$
 P_{\omega,\mu}(v)=\left(\sup \{ \varphi \in \psh(X,\omega) \; ; \; \f \leq v \; \text{ $\mu$-a.e.  on } X \}\right)^* .
 $$
\end{named}

When $\mu$ is a volume form on $X$ and $u$ is quasi-upper-semicontinuous on $X$, the $\omega$-psh upper $(\omega,\mu)$-envelope is just the usual $\omega$-psh envelope (see Proposition  \ref{prop: modified envelope}).  

Recall that a function $v$ is {\it quasi-continuous} 
(resp. {\it  quasi-usc})
if for all $\e>0$,  there exists an open subset $U$ with capacity less than $\vep$ such that the restriction of $v$ to the complement $X\setminus U$ is continuous (resp usc). Quasi-psh functions form a large class
of quasi-continuous functions. We use here the {\it Monge-Amp\`ere capacity} whose definition
is recalled in Section \ref{subsec:cap}.

\smallskip

It is classical that the minimum $\min(u,v)$ of two viscosity supersolutions $u,v$ is again a  viscosity supersolution.
Note however that the minimum of two psh functions is no longer psh.
It  follows nevertheless from (an extension of) \ref{thmA} that $P(\min(u,v))$ is both psh and a pluripotential supersolution.
We extend this observation in Lemma \ref{lem: partition inequality} far beyond the viscosity frame (which deals with continuous densities),
and use it to solve  complex Monge-Amp\`ere equations.
We show in particular the following:

\begin{named}{Theorem C}\label{thmC}
	Let $(X,\omega)$ be a compact K\"ahler manifold and $\mu$ a non-pluripolar Radon measure in some open subset of $X$.
 Assume there exists a finite energy 
 subsolution $u_0 \in {\mathcal E}(X,\omega)$, i.e. such that 
 $$
 (\omega+dd^c u_0)^n \geq e^{u_0} \mu.
 $$
 
 Then the envelope of supersolutions
 $$
 \f:=P\left( \inf \{ \p, \p \in {\mathcal E}(X,\omega) \text{ and } (\omega+dd^c \p)^n \leq e^{\p} \mu\} \right)
 $$
 is the unique pluripotential solution of $ (\omega+dd^c \f)^n = e^{\f} \mu$.
\end{named}

 This result could be restated in the more familiar form: if there exists a subsolution, then there exists a solution.
 Let us emphasize here that the measure $\mu$ is  not necessarily  a Radon measure in all of $X$.
 A particular case of interest in complex differential geometry is when $\mu=fdV_X$ is a volume form
 whose density is smooth and positive in some Zariski open set $\Omega$, but does not belong to $L^1$.
 The existence of a subsolution insures that $f$ does not blow up too fast near $\partial \Omega$
 and is easy to check in many concrete examples: we thus provide an alternative proof of the existence of K\"ahler-Einstein metrics 
 on varieties with negative first Chern class and semi log-canonical singularities,
 a result first obtained by different techniques by Berman-Guenancia \cite{BG14}.
 
\smallskip

On our way to proving the above theorems, we establish several other results of independent interest:
we show in particular that
\begin{itemize}
\item the finite energy functional $I_p$ satisfies a quasi-triangle inequality for any $p>0$
(Theorem \ref{thm: quasi triangle inequality}  largely extends \cite[Theorem 1.8]{BBEGZ11});
\item generalized capacities are all quantitatively comparable (see Theorem \ref{thm: DNL big} which generalizes  \cite[Theorem 2.9]{DNL14b}).
\end{itemize}

 \medskip 
 
 \noindent {\bf Organization of the paper.} 
We introduce the main pluripotential tools in Section \ref{sec: preliminaries}, providing simplifications and extensions of some useful results (see Theorems \ref{thm: quasi triangle inequality} and \ref{thm: DNL big}).
We make a systematic study of (q)psh envelopes in Section \ref{sec:envelope} and establish \ref{thmB}. Inspired by this convergence result we prove \ref{thmA} in Section \ref{sec:viscosity}.  \ref{thmC} is proved in Section \ref{sec:resolutionMA} while other applications are given in Section \ref{sec:applications}.

\medskip

 \noindent {\bf Acknowledgement.}    We warmly thank Robert Berman for fruitful discussions concerning his convergence method.

\section{Preliminaries}\label{sec: preliminaries}

We review recent results in pluripotential theory  and establish a few extensions of the latter that are needed in this paper. 
We refer the reader to \cite{BT82,Kol98,GZ05,BEGZ10} and the references therein for more details.

In the whole paper, $(X,\omega)$ is a compact K\"ahler manifold of dimension $n\in \N^*$, and  $\theta$ is a closed smooth $(1,1)$-form on $X$.

\subsection{Pluripotential theory in big cohomology classes}

A function $u$ is called {\it quasi-plurisubharmonic} on $X$ ({\em qpsh} for short) if in any local holomorphic coordinates it can be written as $u=\rho+\varphi$ where $\rho$ is smooth and $\varphi$ is plurisubharmonic. It is called {\it  $\theta$-plurisubharmonic} ($\theta$-psh for short) if it additionally satisfies 
$$
\theta +dd^c u \geq 0
$$
 in the weak sense of currents. We let $\psh(X,\theta)$ denote the set of all $\theta$-psh functions on $X$ which are not identically $-\infty$.  
 
 By definition the class $[\theta]$ is {\it big} if 
 there exists $\psi\in \psh(X,\theta)$ such that $\theta +dd^c \psi\geq \vep \omega$ for some small constant $\vep>0$.

A function $u$ has {\it analytic singularities} if  it can locally be written as 
\[
u(z) = c \log \sum_{j=1}^k |f_j(z)|^2 + h(z),
\]
where the $f_j's$ are local holomorphic functions and $h$ is  smooth.

By the fundamental approximation theorem of Demailly \cite{Dem92} any quasi-plurisubharmonic function $u$ can be approximated from above by a sequence $(u_j)$ of $(\alpha+\vep_j\omega)$-psh functions with analytic singularities, where $\alpha$ is a closed smooth $(1,1)$-form such that $\alpha +dd^c u\geq 0$. Applying this result to the potential $\psi$ of the K\"ahler current $\theta +dd^c \psi$,
 it follows that  there exists $\theta$-psh functions with analytic singularities. 

Following \cite{Bou04,BEGZ10} the  {\it ample locus} of $\{\theta\}$ is defined to be the set of all $x\in X$ such that there exists a $\theta$-psh function on $X$ with analytic singularities which is smooth in a neighborhood of $x$. 

A $\theta$-psh function $u$ is said to have  {\it minimal singularities} if it is less singular than any other $\theta$-psh function on $X$, more precisely if for any $v\in \psh(X,\theta)$ there exists a constant $C=C(u,v)$ such that $u\geq v-C$. 
The function 
$$
V_{\theta}:=\sup \{ \f \; : \f \in {\rm PSH}(X,\theta) \text{ and } \f \leq 0 \}
$$
has minimal singularities.

If $u_1,...,u_k$ are $\theta$-psh functions with minimal singularities on $X$, then they are locally bounded in the ample locus $\Omega:=\Amp(\{\theta\})$. By the seminal work of Bedford-Taylor \cite{BT76,BT82}  the current $\theta_{u_1}\wedge ...\wedge \theta_{u_k}$ is well defined, and positive in $\Omega$. 
As the total mass is bounded, one can extend it trivially to the whole of $X$. It was proved in \cite{BEGZ10} that the current obtained by this trivial extension is closed. In particular, when $k=n$ and $u_1=...=u_n$ this procedure defines the (non-pluripolar) Monge-Amp\`ere measure of $u$. 

For a general $\theta$-psh function $u$,
the  approximants $u_j:=\max(u,V_{\theta}-j)$  have minimal singularities. One can show that the sequence of positive Borel measures 
${\bf 1}_{\{u > V_{\theta}-j\}}(\theta+dd^c u_j)^n$ is increasing in $j$. Its limit 
$$
\MA_\theta(u):=\lim \nearrow {\bf 1}_{\{u\geq V_{\theta}-j\}}(\theta+dd^c u_j)^n
$$
(in the strong sense of Borel measures) is  {\it  the non-pluripolar Monge-Amp\`ere measure} of $u$.  We denote it by  $\MA(u)$ if no confusion can arrive. 

The  {\it volume} of the class  $[\theta]$  is given by the total mass of the Monge-Amp\`ere measure of $V_{\theta}$
(see \cite{Bou04,BEGZ10}).

A set $E\subset X$ is called pluripolar if locally it is contained in the $-\infty$ locus of some plurisubharmonic function. It was shown in \cite{GZ05} that $E$ is pluripolar if and only if $E\subset \{\phi=-\infty\}$ for some $\phi \in \psh(X,\omega)$. One can replace the K\"ahler form $\omega$ by any big form $\theta$. Indeed if $\theta$ is big then there exists $\psi \in \psh(X,\theta)$ such that $\theta +dd^c \psi\geq \vep \omega$ for some positive constant $\vep$. The function $\psi' :=\vep \phi + \psi$ is $\theta$-psh and its $-\infty$-locus contains $E.$

\subsection{Convergence in capacity} \label{subsec:cap}
Following \cite{BT82,Kol03,GZ05,BEGZ10} we consider the Monge-Amp\`ere capacity with respect to the form $\theta$ ,
\[
{\rm Cap}_{\theta}(E) := \sup \left\{\int_E\MA(u) \ :\ u\in \psh(X,\theta), \ V_{\theta}-1\leq u\leq V_{\theta} \right\}.
\]
A sequence $(u_j)$ is said to converge in capacity  to $u$ if for any $\vep>0$,
\[
\lim_{j\to +\infty}{\rm Cap}_{\theta} (|u_j-u|>\vep) \to 0. 
\] 

The definition does not depend on  $\theta$ because ${\rm Cap}_{\theta}$ and ${\rm Cap}_{\theta'}$ are comparable for any big form $\theta'$. Indeed, it was proved in \cite[Theorem 2.3]{DDL16} that there exists a constant $C=C(\theta,\theta') \geq 1$ such that
\[
C^{-1}{\rm Cap}_{\theta}^n \leq {\rm Cap}_{\theta'} \leq C{\rm Cap}_{\theta}^{1/n}.
\] 

 The following lemma will be used several times in this paper:

\begin{lem}\label{lem: plurifine convergence}
	Assume that a sequence $(u_j) $  in $\psh(X,\theta)$ has uniformly minimal singularities (i.e. there exists $C>0$ such that $u_j\geq V_{\theta}-C$ for all $j$) and converges in capacity to $u$. Then $\MA(u_j)$  weakly converges to $\MA(u)$.  Moreover, if $\varphi_1,\varphi_2$ are quasi-plurisubharmonic functions then 
	\[
	\liminf_{j\to +\infty} {\bf 1}_{\{\varphi_1<\varphi_2\}} \MA(u_j) \geq {\bf 1}_{\{\varphi_1<\varphi_2\}} \MA(u),
	\]
	and 
	$
	\limsup_{j\to +\infty} {\bf 1}_{\{\varphi_1\leq \varphi_2\}} \MA(u_j) \leq {\bf 1}_{\{\varphi_1\leq \varphi_2\}} \MA(u).
	$
\end{lem}

\begin{proof}
The weak convergence of $\MA(u_j)$ to $\MA(u)$ on each relatively compact open subset $U$ of
the ample locus  $\Omega$ follows from \cite[Theorem 1]{X96}. 
The complement $X\setminus U$ can be chosen to have arbitrarily small capacity, as $X\setminus \Omega$ is pluripolar.  This justifies the weak convergence of the Monge-Am\`ere measures on the whole of $X$. The second statement follows since  $\{\varphi_1<\varphi_2\}$ is open and $\{\varphi_1\leq \varphi_2\}$ is closed in the plurifine topology (see \cite{BT87}).
\end{proof}

We now provide a simple criterion insuring that a sequence converges in capacity (recall that convergence in
$L^1$ does not imply convergence of the Monge-Amp\`ere measures \cite{Ceg83}) :

\begin{lem}\label{lem: easy convergence in cap}
Assume that $(u_j) \in PSH(X,\theta)^{\N}$   converges  to $u\in \psh(X,\theta)$ in $L^1(X)$. 
If $u_j \geq u + \vep_j\psi$ for any $j \in \N$, where
 $\psi$ is quasi-plurisubharmonic and $(\vep_j)$  decreases to $0$, then $(u_j)$ converges in capacity towards $u$. 
\end{lem}
The result may be well-known to experts. The (sketch of) the proof is given below for the reader's convenience. 
\begin{proof}
	By multiplying $\omega$ with some big constant $A>0$ we can assume that $u_j,u,\psi$ are $\omega$-psh. We can also assume that all functions involved here are negative. By \cite[Proposition 3.6]{GZ05}  we have
	\[
	\capo(\varphi<-t)\leq \frac{C}{t}, \forall t>0,
	\]   
	where $C>0$ depends on $\int_X \varphi \ \omega^n$. It thus suffices to prove that $\max(u_j,-t)$ converges to $\max(u,-t)$ in capacity for any $t>0$. For this reason we can now assume  that $u_j,u$ are bounded. 
	
Classical arguments  show that the sequence 
	$v_j:= \max(u_j,u)$ converges in capacity to $u$. Thus it remains to show that for a fixed $\vep>0$, 
	\[
	\lim_{j\to +\infty}\capo (u_j \leq u - \vep) =0.
	\]
The proof is completed by  noting that
	\[
	\{u_j\leq u-\vep \} \subset \left\{\psi \leq -\frac{\vep}{\vep_j} \right\}. 
	\]
\end{proof}

\subsection{Finite energy classes}
Finite energy classes have been introduced in \cite{GZ07}  and further studied in \cite{BEGZ10,DDL16}. 
They provide a very convenient frame to study convergence properties of $\theta$-psh potentials.

\begin{defi}
A function $u\in \psh(X,\theta)$ belongs to $\Ec(X,\theta)$ if the total mass of $\MA(u)$ is equal to the volume of $\theta$. 
\end{defi}

 It is proved in \cite{DDL16} building on \cite{Dar13} that the singularity type of functions in $\Ec(X,\theta)$ is the same as that of $\theta$-psh functions with minimal singularities (in particular they have the same Lelong numbers at every point). 
 
 \begin{defi}
 For $p>0$, the set $\Ec^p(X,\theta)$ consists of functions $u$ in $\Ec(X,\theta)$ such that $\int_X |u-V_{\theta}|^p \MA(u)<+\infty$.   
 \end{defi}
 
More generally, a weight is a smooth increasing function $\chi: \mathbb{R}\rightarrow \mathbb{R}$ such that $\chi(-\infty)=-\infty$.  The class $\mathcal{E}_{\chi}(X,\theta)$ consists of functions $u\in \mathcal{E}(X,\theta)$ such that 
\[
\int_X |\chi(u)| \MA(u)<+\infty.
\] 

The (unnormalized) Monge-Amp\`ere energy of a $\theta$-psh function $u$ with minimal singularities is
\[
E(u):= \frac{1}{n+1}\sum_{k=0}^{n}\int_X(u-V_{\theta}) (\theta+dd^c u)^k \wedge (\theta+dd^c V_{\theta})^{n-k}. 
\]
One extends the definition to arbitrary $\theta$-psh functions by 
\[
E(\varphi):=\inf\{E(u) \ :\ u\in \psh(X,\theta) \ \textrm{with minimal singularities}\ , \ u\geq \varphi \}.
\]

It is shown in \cite{BEGZ10,BBGZ13} that $\varphi \in \Ec^1(X,\theta)$ if and only if $E(\varphi)$ is finite. Moreover $E$ is increasing, concave along affine curves in $\Ec^1(X,\theta)$ and upper semicontinuous with respect to the $L^1$-topology.

We need the following generalization of \cite[Proposition 2.10]{BEGZ10}:

\begin{lem}\label{lem: BEGZ prop 2.10}
	Fix $p>0$. Assume that $(u_j)$ is a decreasing sequence of $\theta$-psh functions with minimal singularities and $\varphi\in \Ec^p(X,\theta)$ is such that
	\[
	\sup_{j\in \N^*}\int_X|u_j-\varphi|^p\MA(u_j) <+\infty.
	\]
	Then $u:=\lim_{j}u_j$ belongs to $\Ec^p(X,\theta)$.
\end{lem}

\begin{proof}
If $p=1$ we can use the concavity of the Monge-Amp\`ere energy $E$. It follows from the assumption and from \cite[proposition 2.1]{BBGZ13} that 
	\[
	E(u_j) -E(\varphi) \geq \int_X (u_j-\varphi) \MA(u_j) \geq -C,
	\]
	for some   $C>0$. Hence Proposition 2.10 in \cite{BEGZ10} gives the conclusion. 
	
	We now deal with the general case $p>0$.
If $\theta$ is additionally semipositive, and $p>1$, it follows from \cite{DNG16} (see \cite{Dar14,Dar15} for the K\"ahler case) that the   functional  
\[
I_p(u,v) := \int_X |u-v|^p(\MA(u)+\MA(v)) 
\] 
satisfies a quasi-triangle inequality. From this observation and the assumption we get a
uniform  bound on $I_p(u_j,0)$, hence $\lim_j u_j$ belongs to $\Ec^p(X,\theta)$. 
For an arbitrary $\theta$, and $p>0$, we use the quasi-triangle inequality from Theorem \ref{thm: quasi triangle inequality} below.
\end{proof}

The following quasi-triangle inequality insures that $I_p$ induces a uniform structure on $\Ec^p(X,\theta)$:
\begin{thm}\label{thm: quasi triangle inequality}
	Let $p>0$ and $u,v,\varphi \in \psh(X,\theta)$ with minimal singularities. There exists a uniform constant $C>0$ depending on $p,n$ such that 
	\[
	I_p(u,v) \leq C(I_p(u,\varphi) + I_p(v,\varphi)). 
	\]
\end{thm}

As indicated above, this result was already known when $\theta$ is semi-positive. The proof given here covers the general case
and also provides a simpler proof of the previous cases.

\begin{proof}
Observe that for a positive measure $\mu$ and a non-negative measurable function $f$ on $X$ the integral $\int_X f d\mu$ can be expressed as 
\begin{equation}
	\label{eq: quasi-triangle inequality 0}
	\int_X f d\mu = \int_0^{+\infty} \mu(f>t) dt.
\end{equation}
Using this and a change of variable $t=(2s)^p$ we can write 
\begin{eqnarray}
	\int_X |u-v|^p \MA(u) &=& \int_0^{+\infty} \MA(u) (|u-v|^p>t) dt \nonumber \\
	&=& 2^p p\int_0^{+\infty} s^{p-1}\MA(u) (|u-v|>2s) ds. \label{eq: quasi-triangle inequality 1}
\end{eqnarray}
For $s>0$ we observe that the following inclusion holds 
\[
(\varphi-s\leq u <v-2s) \subset \left(\varphi < \frac{u+2v}{3} -\frac{s}{3}\right).  
\]
Hence, using this and the trivial inclusion 
\[
(u<v-2s)\subset (u<\varphi-s) \cup (\varphi-s\leq u <v-2s)
\]
we can write 
\begin{eqnarray}
	\MA(u) (u<v-2s) &\leq &  \MA(u) (u<\varphi-s)\nonumber \\
	&+ & \MA(u)\left(\varphi < \frac{u+2v}{3} -\frac{s}{3}\right). \nonumber 
\end{eqnarray}
Using this, the inequality $\MA(u)\leq 3^n \MA((u+2v)/3)$ and the comparison principle, we obtain 
	\begin{equation*}
		 \MA(u) \left(\varphi < \frac{u+2v}{3} -\frac{s}{3}\right) \leq  3^n  \MA(\varphi) \left(\varphi <\frac{u+2v}{3} -\frac{s}{3}\right).
	\end{equation*}
The comparison principle also yields 
\[
\MA(u) (v<u-2s) \leq \MA(v)(v<u-2s). 
\]
We then use the same argument as above to get
\begin{multline*}
	\MA(u) (|u-v|>2s) \leq \MA(u)(u<\varphi-s) + \MA(v)(v<\varphi-s) \\
	+  3^n  \MA(\varphi) \left(\varphi <\frac{u+2v}{3} -\frac{s}{3}\right) + 3^n  \MA(\varphi) \left(\varphi <\frac{v+2u}{3} -\frac{s}{3}\right).
\end{multline*}
From this and \eqref{eq: quasi-triangle inequality 0} we thus obtain 
\begin{multline*}
	\int_0^{+\infty} ps^{p-1} \MA(u) (|u-v|>2s) \leq \int_{0}^{+\infty} p s^{p-1}\MA(u)(u<\varphi -s)ds \\
	+  \int_{0}^{+\infty} p s^{p-1}\MA(v)(v<\varphi -s) ds + 3^n \int_0^{+\infty} ps^{p-1}  \MA(\varphi) \left(\varphi <\frac{u+2v}{3} -\frac{s}{3}\right)ds\\
	+ 3^n \int_0^{+\infty} ps^{p-1}  \MA(\varphi) \left(\varphi <\frac{2u+v}{3} -\frac{s}{3}\right)ds \\
	\leq \int_X |u-\varphi|^p \MA(u) + \int_X |v-\varphi|^p \MA(v)\\
	+ 3^{n+p}\int_X \left | \varphi-\frac{u+2v}{3}\right |^p \MA(\varphi) 
	+ 3^{n+p}\int_X \left | \varphi-\frac{2u+v}{3}\right |^p \MA(\varphi).
\end{multline*}
Using this, \eqref{eq: quasi-triangle inequality 1} and the elementary inequality 
\[
(a+b)^p \leq \max(2^{p-1},1) (a^p +b^p)
\] 
for $a,b>0, p>0$,  we arrive at 
\begin{multline*}
	\int_X |u-v|^p \MA(u) \leq 2^p \int_X |u-\varphi|^p \MA(u) + 2^p \int_X |v-\varphi|^p \MA(v)\\
	+ 3^n 2^{2p+1} \int_X\left( |\varphi-u|^p + |\varphi-v|^p\right) \MA(\varphi). 
\end{multline*} 
We then proceed similarly to treat the term $\int_X |u-v|^p \MA(v)$ and arrive at the conclusion.
\end{proof}

\subsection{Convergence in energy}

Monotone convergence implies convergence in capacity, which insures convergence of the Monge-Amp\`ere operator, as indicated above. 
A stronger notion of convergence has been introduced in \cite{BBEGZ11}:

\begin{defi}
A sequence $(u_j) \in \Ec^1(X,\theta)^{\N}$  converges in energy to $u \in \Ec^1(X,\theta)$ if
\[
0\leq I(u_j,u) := \int_X (u_j-u) (\MA(u) -\MA(u_j)) \to 0. 
\]
\end{defi}

The functional $I$ is well adapted to {\it normalized} potentials. For unnormalized ones, one should use
$$
I_1(u_j,u)=\int_X |u_j-u| \left(\MA(u_j)+\MA(u) \right).
$$
We let the reader check that if a sequence $(u_j) \in \Ec^1(X,\theta)^{\N}$ is normalized by $\sup_X u_j=0$,
it converges in energy to $u$ if and only if $I_1(u_j,u) \rightarrow 0$.

\smallskip

It is shown in \cite{BBEGZ11} that convergence in energy implies convergence in capacity. 
The converse is however not true  as the following example shows:

\begin{exa}
Assume  $\omega$ is a K\"ahler form and $\varphi$ is a $\omega$-psh function which 
locally near a point $z_0$ (identified with zero in a local coordinate chart) is defined by $\varphi(z)= a\log |z|^2$, for some 
$a>0$ small enough. Then the $\omega$-psh function $\varphi$  has a Dirac Monge-Amp\`ere mass at $0$. 

Define 
\[
u_j:=\frac{1}{j} \varphi_j\ ; \ \varphi_j:=\max(\varphi,-j^{n+2}).
\]
We let the reader check  that $u_j$ converges  to $u=0$ in capacity.
For any $j\in \mathbb{N}$ we have $\MA(\varphi_j)(\varphi\leq -j^{n+2})\geq c>0$. In fact, we only need to know that $\varphi\not \in \Ec(X,\omega)$ and $c$ is the loss of the total mass of the non-pluripolar Monge-Amp\`ere measure of $\varphi$.  Then the energy of $u_j$ is computed by 
\[
\int_X |u_j| \MA(u_j) \geq \int_{\{\varphi\leq -j^{n+2}\}} |u_j|\MA(u_j)\geq j^{n+1}\int_{\{\varphi\leq -j^{n+2}\}}j^{-n}\MA(\varphi_j)\geq j c.
\]
This shows that $(u_j)$ does not converge to $0$ in energy.
\end{exa}


The next result says that these convergences are equivalent if the sequence is bounded from below by a finite energy function. 

\begin{pro}
\label{lem: capacity implies energy}
Let  $u_j,u \in \Ec^p(X,\theta)$, $p>0$, and assume that $u_j$ converges in capacity to $u$.  If there exists $\varphi\in \Ec^p(X,\theta)$ such that $u_j\geq \varphi, \forall j$, then $u_j$ converges  to $u$ with respect to the quasi-distance $I_p$, i.e. 
\[
\int_X |u_j-u|^p(\MA(u_j)+ \MA(u)) \to 0. 
\] 
\end{pro}

Conversely one can show that if $u_j$ converges in energy to $u$, then up to extracting and relabelling, there exists
$\varphi\in \Ec^p(X,\theta)$ such that $u_j\geq \varphi, \forall j$.

\begin{proof}
The idea of the proof is essentially contained in \cite{GZ07}, \cite{BDL15}. It costs no generality to assume that $u_j\leq 0$ which also implies that $u\leq 0$. By dominated convergence it suffices to prove that 
\[
\lim_{j\to +\infty} \int_X |u_j-u|^p  \MA(u_j) =0. 
\]
By using a truncation argument as in \cite[Theorem 2.6]{GZ07} one can show that $\MA(u_j)$ converges weakly to $\MA(u)$. 
We claim that
\[
\lim_{s\to +\infty} \int_{\{u_j\leq u-s\}} |u_j-u|^p \MA(u_j) =0
\]
uniformly in $j$. Indeed the comparison principle in $\Ec(X,\theta)$ shows that
{\small
\begin{eqnarray*}
\int_{\{u_j\leq u-s\}} |u_j-u|^p \MA(u_j) &=&p\int_s^{+\infty} t^{p-1}\MA(u_j) (u_j< u-t)dt\\
&\leq & 2^n p  \int_{s}^{+\infty} t^{p-1} \MA\left(\frac{u_j+u}{2}\right) \left(\varphi<\frac{u_j+u}{2} -\frac{t}{2}\right)dt\\
&\leq & 2^np \int_{s}^{+\infty} t^{p-1}\MA(\varphi) (\varphi<V_{\theta}-t/2)dt. 
\end{eqnarray*}
}
The last term converges to zero as $s\to +\infty$ because $\varphi\in \Ec^p(X,\theta)$. Thus the claim is proved. One can also show that 
\[
\lim_{s\to +\infty} \int_{\{u \leq  u_j-s\}} |u_j-u|^p \MA(u_j) =0
\]
uniformly in $j$. 

Therefore, it remains to show that 
\[
\lim_{j\to +\infty} \int_{\{|u_j-u|\leq s\}} |u_j-u|^p \MA(u_j) =0
\]
for any fixed $s>0$. But the latter  follows from Theorem \ref{thm: DNL big} below.
\end{proof}

\subsection{Generalized capacities}
Let $(X,\omega)$ be a compact K\"ahler manifold of dimension $n$ and fix $\theta$ a smooth closed $(1,1)$-form such that $[\theta]$ is big. 

\begin{defi}
For $\varphi,\psi \in \psh(X,\theta)$ such that $\varphi<\psi$ the $(\varphi,\psi)$-capacity is defined by
\[
\capfp(E) := \sup\left \{ \int_E \MA(u) \, :\, u\in \psh(X,\theta), \varphi\leq u\leq \psi\right \}, \ E\subset X. 
\] 
\end{defi}

When $\theta$ is K\"ahler these generalized Monge-Amp\`ere capacities were introduced in \cite{DNL14a,DNL14b}.  
When $\varphi=\psi-1=V_{\theta}-1$ we recover the Monge-Amp\`ere capacity ${\rm Cap}_{\theta}$. 

It was observed by Di Nezza \cite{DN15} that for each $t\geq 1$, setting $\varphi=V_{\theta}-t, \psi=V_{\theta}$, one has
\[
{\rm Cap}_{\theta}(E) \leq \capfp(E) \leq t^n {\rm Cap}_{\theta}(E), \ \forall E\subset X. 
\]
When $\theta$ is K\"ahler all $(\varphi,\psi)$-capacities with $\varphi\in \Ec(X,\theta)$ are comparable \cite{DNL14b}. We generalize this result in the context of big classes. 

\begin{thm} 	\label{thm: DNL big}
	Assume that $0\geq \varphi\in \Ec_{\chi}(X,\theta)$ for some convex weight $\chi: \R^{-}\rightarrow \R^{-}$.
 Then there exists a continuous function $F_{\chi}:\R^{+}\rightarrow \R^{+}$   such that $F_{\chi}(0)=0$ and 
 for all $E\subset X$,
	\[
	\capfp(E) \leq F_{\chi}({\rm Cap}_{\theta}(E)).
	\]
\end{thm}

\begin{proof}
We can assume without loss of generality that $\sup_X \psi=0$. Then $\psi \leq V_{\theta}$, hence
$$
\capfp(E) \leq {\rm Cap}_{\f,V_{\theta}}(E).
$$
It thus suffices to treat the case $\p=V_{\theta}$.

	Fix $E\subset X$ a non-pluripolar Borel subset. We can assume $\captheta(E)<1$.  Fix a constant $t>1$, a 
	$\theta$-psh function $u$ such that $\varphi\leq u \leq V_{\theta}$ and set $u_t:=\max(u,V_{\theta}-t)$. 
	Observe that $V_{\theta}-t \leq u_t \leq V_{\theta}$ hence
	$$
	V_{\theta}-1 \leq t^{-1}u_t +(1-t^{-1}) V_{\theta} \leq V_{\theta}.
	$$
	It follows that 
	for all $E \subset X$,
	\[
	\int_E \MA(u_t) \leq 
	t^n \int_E \MA(t^{-1}u_t +(1-t^{-1}) V_{\theta}) \leq t^n \captheta(E). 
	\] 
	On the other hand by the comparison principle we also have 
		\begin{eqnarray*}
		\int_{\{u \leq V_{\theta}-t\}} \MA(u) &\leq & 2^n \int_{\{\varphi \leq (u+V_{\theta}-t)/2\}}  \MA((u+V_{\theta}-t)/2)\\
		&\leq & 2^n \int_{\{\varphi \leq V_{\theta} -t/2\}} \MA(\varphi)\\
		&\leq & \frac{2^n}{|\chi(-t/2)|} \int_X |\chi(\varphi-V_{\theta})|\MA(\varphi). 
	\end{eqnarray*}
	Thus there is a constant $C>0$ depending on $E_{\chi}(\varphi)$ such that 
	\[
	\int_E \MA(u) \leq t^n \captheta(E) + \frac{C}{|\chi(-t/2)|}. 
	\]
We can choose $t= (\captheta(E))^{-1/(n+1)}$ and get the conclusion.  
\end{proof}

\subsection{Degenerate Monge-Amp\`ere equations}

Given a non pluripolar positive Radon measure $\mu$ on $X$, it is useful to consider the equations
\begin{equation}
	\MA(\varphi) =e^{\lambda \varphi} \mu, \label{eq: MAlambda} 
\end{equation}
where $\lambda\in \mathbb{R}$ is a constant. When $\lambda=0$ a normalization condition $\mu(X)=\vol(\theta)$ should be imposed. 
We will use the following result obtained by a variational method in \cite{BBGZ13}: 

\begin{thm}\label{thm: BBGZ13}
	Assume that $\lambda>0$. Then there exists a unique $\varphi\in \Ec^1(X,\theta)$ solving \eqref{eq: MAlambda}. 
\end{thm}

The main idea of the proof in \cite{BBGZ13} is to maximize the functional 
$$
\mathcal{F}_{\lambda}(\p)=E(\p)+\frac{1}{\lambda} \log \left (\int_X e^{\lambda \p} d\mu \right).
$$

The continuity and the coercivity of the functional $\mathcal{F}_{\lambda}$  is automatic when $\lambda>0$.  
When $\lambda=0$ the problem is more subtle and the resolution so far relies on  a regularity result of $V_{\theta}$ obtained in \cite{BD12} : one uses the fact that $\MA(V_{\theta})$ has bounded density with respect to Lebesgue measure. The latter is a direct consequence of our \ref{thmA} (see Proposition \ref{prop: weak BD12}). An alternative proof of Proposition \ref{prop: weak BD12} has been given recently in \cite{Ber13}, \cite{DDL16} using ideas from the viscosity theory in \cite{EGZ11}.


\smallskip

We will also need the {\it  domination principle} which, in the context of big classes,
 was first established in \cite{BEGZ10} for $\theta$-psh functions with minimal singularities. 
 The result still holds for functions in $\mathcal{E}(X,\theta)$ as follows from an argument due to Dinew
 (see \cite{BL12,DDL16}). 
 
\begin{pro}\cite{BEGZ10,DDL16}\label{prop: domination principle}
	Let $\varphi,\psi$ be $\theta$-psh functions on $X$ and assume that $\varphi\in \mathcal{E}(X,\theta)$. If $\MA(\varphi)(\varphi<\psi)=0$ then $\varphi\geq \psi$ on $X$. 
\end{pro} 
 
We will occasionally use the following version of the comparison principle, which follows from the domination principle (see  \cite{DDL16}):  

\begin{pro}\label{prop: comparison principle exponential}
	Let $\varphi,\psi \in \Ec(X,\theta)$, $\mu$ be a non-pluripolar positive measure and   $f$ be  a Borel measurable function  on $X$ such that 
	\[\MA(\varphi) \geq e^{\beta \varphi} e^{-f} \mu \ ; \ \MA(\psi) \leq e^{\beta \psi} e^{-f} \mu,\]
	where $\beta>0$ is a constant. Then $\varphi\leq \psi$. 
\end{pro}

\section{Envelopes} \label{sec:envelope}

Upper envelopes are classical objects in Potential Theory.
They were used  in the Perron method for solving the Dirichlet problem for the Poisson equation as well 
as the free boundary problems for the Laplace operator.

Upper envelopes of psh functions were considered   by Bremermann, Walsh, Siciak,  
and Bedford and Taylor to solve the Dirichlet problem for the complex Monge-Amp\`ere equation in strictly pseudo-convex domains
(see \cite{BT76}).

We consider here envelopes of quasi-psh functions on compact K\"ahler manifolds, following
\cite{GZ05,  Ber13}.

\subsection{Usual envelopes}

Let $X$ be a compact K\"ahler manifold of complex dimension $n$ and let $\theta$ be a closed smooth real $(1,1)$-form on $X$ whose cohomology class is big.

\begin{defi} \label{def:usual}
Given a Lebesgue measurable function $h:X \rightarrow \R$  which is bounded from below, we define the $\theta$-psh envelope of $h$ as follows
$$
P_{\theta}(h) := \left(\sup \{ u \in \psh (X,\theta) ; u \leq h  \, \, \text{in} \, \, X\}\right)^*,
$$
where the star means that we take the upper semi-continuous regularization. 
\end{defi}

We will also denote by $P(h)$ the envelope $P_{\theta}(h)$ if no confusion can arrive. 
In this section we start a systematic study of these envelopes.

\smallskip

When $h = - {\bf 1}_E$ is the negative characteristic function of a subset $E$ then $P(h) = h_E^*$ is the so called relative extremal function of $E$ \cite{GZ05}. 

When $h = 0$ then $P(0) = V_{\theta}$ was introduced in \cite{DPS01} as an example of a $\theta$-psh function with minimal singularities.
When $h$ is smooth Berman and Demailly have shown in \cite{BD12} that $P(h)$ has locally bounded laplacian in $\Amp(\theta)$. 
In particular $V_{\theta}=P(0)$ has locally bounded laplacian 
 and the Monge-Amp\`ere measure of $V_{\theta}$ can be described as 
\[
\MA(V_{\theta}) = {\bf 1}_{\{ V_{\theta}=0\} }\theta^n.
\]
In the case when the class $[\theta]$ is big and nef, a PDE proof of this result was given by Berman in \cite{Ber13}. The fundamental observation of Berman is that the envelope can be obtained as the limit of solutions to a one-parameter family of complex Monge-Amp\`ere equations (this idea has been recently used in \cite{LN15}, \cite{KN16}, \cite{BL16}).  By establishing a uniform laplacian estimate for  this family of solutions  Berman showed that $P_{\theta}(h)$ has locally bounded laplacian in the ample locus. 

In particular, when the cohomology class of $\theta$ is K\"ahler then $P_{\theta}(h)$ has bounded laplacian on $X$.
The optimal $\mathcal{C}^{1,1}$ regularity  of $P_{\theta}(h)$, conjectured by Berman,  has recently been confirmed by Tosatti \cite{Tos17} using the $\mathcal{C}^{1,1}$ estimate in \cite{CTW16} and the convergence method of Berman \cite{Ber13}.  

\smallskip

We need here to study these envelopes for functions $h$ that are less regular. 
When $h=0$ on $E$ and $+\infty$ on $X\setminus E$ then $P(h)$ is the global extremal $\theta$-psh function of $E$ that was considered in \cite{GZ05,BEGZ10}.  It follows from \cite{GZ05, BEGZ10} that if $h$ is finite on a non-pluripolar set then $P (h)$ is a well-defined $\theta$-psh function on $X$.

\begin{pro}\label{prop: basic property of envelope}
\text{ }

1. If $h:X \rightarrow \R$ is a bounded measurable function then $P(h) $ is a $\theta$-psh function with minimal singularities
which satisfies  $P(h) \leq h$ quasi everywhere in $X$. Moreover
$$
P (h) := \sup \{ u \in PSH (X,\theta) :  u \leq h , \, \, \text{quasi everywhere in} \, \, X\}.
$$

\smallskip
 
 2. If $(h_j)$ is a decreasing sequence of bounded measurable functions   which converge pointwise to $h$ in $X$,
 then $P(h_j)$ decreases to $P( h)$ in $X$.
 
 \smallskip
 
 3. If $(h_j)$ is an increasing sequence of bounded quasi-lsc functions  converging pointwise to $h$,
 then $P (h_j)$ increases to $P(h)$  quasi everywhere.
 
\smallskip 
 
 4. If $h$ is continuous and $\theta$ is K\"ahler then $P_\theta (h)$ is continuous in $X$.
\end{pro}

In these statements quasi everywhere means outside a pluripolar set.
\begin{proof}
Assume that $h$ is a bounded Lebesgue measurable function on $X$. The fact that $P(h)$ is a $\theta$-psh function with minimal singularities follows directly from the definition.  The set $\{x\in X : P(h)(x) > h(x)\}$ is negligible. It follows from a classical result in pluripotential theory \cite{BT82} that negligible sets are pluripolar, thus $P(h)\leq h$ quasi every where on $X$. 

We now prove the identity in the first statement. Let $\varphi$ denote the function on the right-hand side. It is obvious that $P(h) \leq \varphi$. As a countable union of pluripolar sets is also pluripolar, by Choquet's lemma and the same argument as above we see that $\varphi\leq h$ quasi everywhere on $X$. The equality follows if we can show that $P(h_1)=P(h_2)$ for two bounded functions such that $\{h_1\neq h_2\}$ is pluripolar. Indeed, the set
\[
E := \{P(h_1) >h_1\} \cup \{h_1\neq h_2\}
\]
is also pluripolar.  Hence there exists $\phi \in \psh(X,\theta)$ such that $\phi=-\infty$ on $E$. 
Now for any $\lambda\in (0,1)$  the function $\lambda \phi + (1-\lambda)P(h_1)$ is $\theta$-psh on $X$ and 
bounded from above by  $h_2$. Letting $\lambda \to 0$ one sees that $P(h_1)\leq P(h_2)$ off a pluripolar set, hence the inequality holds everywhere. Conversely one can show that $P(h_2)\leq P(h_1)$, completing the proof of the claim, hence the first statement is proved.

The second statement is straightforward. We now prove the third one. Assume that $(h_j)$ is a sequence of bounded quasi-lsc functions that increase pointwise to a bounded function $h$. Then $P(h_j)$ also increase quasi everywhere to some $\varphi\in \psh(X,\theta)$ with minimal singularities. One observes immediately that $\varphi\leq P(h)$. 
It follows from  Lemma \ref{lem: ort} that $\MA(P(h_j))$ vanishes in $\{P(h_j)<h_j\}$. As $P(h_j)\leq \varphi$ and $h_j\nearrow h$ it follows that $\MA(\varphi)$ also vanishes in $\{\varphi<P(h)\}$. 
The domination principle (Proposition \ref{prop: domination principle})
insures that $\varphi=P(h)$.  

One can prove the last statement by approximation. Let $(h_j)$ be a sequence of smooth functions on $X$ converging uniformly to $h$. It follows from \cite{Ber13} that $P(h_j)$ has bounded laplacian, in particular it is continuous. As $P(h_j)$ uniformly  converges to $P(h)$ the conclusion follows. 
\end{proof}

 The following result is an analogue of the corresponding result of Bedford and Taylor \cite[Corollary 9.2]{BT82} :

\begin{lem} \label{lem: ort} 
Let $h:X \rightarrow \R$ be a bounded Lebesgue measurable function on $X$ and let $L (h)$ be the lower semi-continuity set of $h$.
Then $\MA(P(h))$ puts no mass on the set $L (h) \cap \{P(h) < h\}$.
In particular if  $\capo^* (X \setminus L (h)) = 0$, then
$$
\int_X (P(h) - h) \MA(P(h)) = 0.
$$
\end{lem}

\begin{proof} 
The proof proceeds as in the classical case using a balayage argument (see (\cite[Corollary 9.2]{BT82}).
We repeat it here for the convenience of the reader.  For notational convenience we set $\hat{h}:=P(h)$. Fix a point $x_0 \in L (h) \cap \{\hat h < h\}$.  Observe that by lower semi-continuity at $x_0$, the point $x_0$ lies in the interior of the set $\{\hat h < h\}$. Indeed, fix $\delta > 0$ such that $\hat h (x_0) - h (x_0) \leq - 2 \delta$. By upper semi-continuity of $\hat h$ and lower semi-continuity of $h$ at $x_0$ there exists a small ball $B$ of center $x_0$ such that 
$\max_{\bar B} \hat h  < \min_{\bar{B}}  h  - \delta$. 

Let $\rho$ be a smooth  local potential of $\theta$ in a neighbourhood $D$ of $\bar B$. Shrinking the ball if necessary we can assume that   ${\rm osc}_{\bar B} \rho < \delta$.
Then $u := \hat h + \rho$ is psh in $D$. By Bedford and Taylor (see \cite[Proposition 9.1]{BT82}) there exists a psh function $v$ in $D$ such that $v = u $ in $D \setminus B$,
$v \geq u$ in $D$ and $(dd^c v)^n  = 0$ in $B$.

Since $v = u$ in $\partial B$, the comparison principle insures $ v \geq u$ in $\bar B$. On the other hand on $\partial B$, we have 
$$v = \hat h + \rho  \leq  \max_{\bar B} \hat{h} + \max_{\bar B}\rho. $$
By the classical maximum principle we get $v \leq \max_{\bar B} \hat{h} + \max_{\bar B}\rho $ in $\bar B$, hence 
$v - \rho \leq h - \delta + {\rm osc}_{\bar B} \rho  \leq h$ in $B$.

Therefore since $v - \rho = \hat h$ in $\partial B$, the function $w$ defined by $w := v - \rho$ in $B$ and $w = \hat h$ in $X \setminus B$ is $\theta$-psh in $X$ and satisfies $w \leq h$ in $B$ and $w = \hat h \leq h$ quasi everywhere in $X \setminus B$.
This yields $ w \leq \hat h$ in $X$.

By construction we have $w := v - \rho \geq u - \rho = \hat h$ in $B$ and then $w = \hat h$ in $B$ and  $(\theta + dd^c \hat h)^n = (dd^c v)^n = 0$ in $B$.
\end{proof}

The result above extends to any function $h$ which is quasi lower semi-continuous in the sense that for any $\delta > 0$ there exists a compact set $K \subset X$ such that  $\capo (X \setminus K) \leq \delta$ and the restriction $h|K$ is a lower semi-continuous function.

We need the following fact which follows  from the Tietze-Urysohn lemma.

\begin{lem} 
Let  $h$ be a quasi lower semi-continuous function in $X$. Then there exists a decreasing  sequence $(h_j)$ of lower semi-continuous functions in $X$ which converges to $h$  in capacity and quasi everywhere in $X$.
\end{lem}

\begin{proof} By definition there exists a sequence of compact sets $(K_{\ell})$ such that $\capo (X \setminus K_\ell) \leq 
2^{- \ell}$  and the restriction $h|K_{\ell}$ is a lower semi-continuous function in $K_{\ell}$. 
Take  $\tilde K_j := \cup_{1 \leq \ell \leq j} K_{\ell}$ instead of $(K_j)$, we can assume that the sequence $(K_j)$ is increasing.

Since a lower semi-continuous function on a compact set is the limit of an increasing squence of continuous functions, it follows from  the Tietze-Urysohn Lemma, that there exists a lower semi-continuous function $H_j$ in $X$ such that $H_{j}|K_{j} = h|K_{j}$. Alternatively one can extend $h$ to be $+\infty$ on $X\setminus K_j$.
  
Let $h_j := \sup \{H_{\ell} ; \ell \geq j\}$. Then $(h_j)$ is a  decreasing sequence of lower semi-continuous functions in $X$ such that $h_j = h$ in $K_j$, hence it converges to $h$ in $F := \cup_j K_j$. Since $\capo^* (X \setminus F) = 0$ it follows that $h_j$ converges to $h$ quasi everywhere in $X$.

 We claim that  $(h_j)$ converges to $h$ in capacity. Indeed let $\delta > 0$ be fixed and set for $j \in \N$, $E_j := \{ x \in X ; h_j\geq h + \delta\}$. Since $h_j = h$ in $K_j$, it follows that $E_j \subset X \setminus K_j$. Hence $\capo^* (E_j) \leq \capo^* (X \setminus K_j) \leq 2^{- j}$, and then  
 $$
 \lim_{j \to +\infty} \capo^*(\{ x \in X \, :\, h_j\geq h + \delta\}) = 0,
$$
which proves our claim.
\end{proof}

\begin{pro}\label{prop: ort}
Let $h$ be a bounded quasi lower semi-continuous function in $X$ and set $\hat{h}=P_{\theta}(h)$ the $\theta$-psh envelope of $h$.
Then $\MA_{\theta} (\hat h)$ puts no mass on the set $\{\hat h < h\}$ i.e. 
$$
\int_X (\hat h - h) MA(\hat h) = 0.
$$
\end{pro}

\begin{proof} By the previous lemma there exists a decreasing sequence $(h_j)$ of lsc functions in $X$ such that
$(h_j)$ converges to $h$ in capacity. 

From Lemma~\ref{lem: ort} we know that $\int_X (h_j-\hat h_j) \MA(\hat h_j) = 0$ for any $j$. We also know by Lemma \ref{prop: basic property of envelope} that $(\hat h_j )$ decreases to $\hat h$.  In particular the convergence holds in energy, hence 
$$
\lim_{j\to +\infty}\int_X \hat{h}_j \MA(\hat h_j) =\int_X \hat{h} \MA(\hat{h}). 
$$
On the other hand  the functions $h_j$ are lower semi-continuous, uniformly bounded and converge to $h$ in capacity, hence (see \cite{GZbook, DDL16})
$$
\int_X h \MA(\hat h)  \leq   \liminf_{j \to + \infty} \int_X h \MA(\hat h_j)= \liminf_{j \to + \infty} \int_X h_j \MA(\hat h_j). 
$$
This implies that $\MA(\hat{h})$ puts no mass on the set $\{\hat{h}<h\}$.  
\end{proof}

 The orthogonal relation $\int_X (P(u)-u) \MA(P(u))=0$ does not hold in general, as the following example shows:

\begin{exa}\label{exa: orthogonal relation false}
Assume $X=\C\mathbb{P}^n$ is the complex projective space equipped with the Fubini-Study metric $\theta=\omega_{FS}$.
Let $B$ denote  the unit ball in $\C^n \subset \C\mathbb{P}^n$ and
consider $u$ to be $-1$ on $B$ and $0$ elsewhere.  Then $P(u)$ is the relative extremal function (see \cite{GZ05}) of $B$ which 
	takes values $-1$ on the boundary $\partial B$: we let the reader check  that for $z \in \C^n$,
	$$
	P(u)(z)+\log \sqrt{1+|z|^2}=\max \left\{ \log \sqrt{1+|z|^2}-1; \log|z|+\frac{\log 2}{2}-1 \right\},
	$$
thus $\MA(P(u))$ does not vanish on $\partial B \subset \{P(u)<u\}$. 
\end{exa}

Proposition \ref{prop: ort} generalizes to any upper bounded Borel function $h$  which admits a subextension $\psi \in \mathcal E (X,\theta)$ i.e. $\psi \leq h$ in $X$. 

\begin{thm} 
Let $h$ be a quasi lower semi-continuous function bounded from above in $X$.
Assume there exists $\psi \in \mathcal E (X,\theta)$ s.t. $\psi \leq h$ in $X$. Then 

1. $\hat h = P_{\theta}(h) \in \mathcal E (X,\theta)$ and $\hat h \leq h$ quasi everywhere in $X$.

2. $\MA_\theta (\hat h)$ puts no mass on the set $\{\hat h < h\}$.
\end{thm}

\begin{proof} 
We may assume that $h \leq 0$ in $X$.  Since $\psi\in \Ec(X,\theta)$  
there exists a convex increasing weight $\chi : \R^- \to \R^-$ such that $\psi \in \mathcal E_{\chi} (X,\theta)$ (see \cite{GZ07,BEGZ10}). 
Since $\psi \leq h$, we conclude from the definition that $\psi \leq \hat{h}$ in $X$, hence $\hat{h} \in \mathcal E (X,\theta)$ by \cite{GZ07,BBEGZ11}.  

Set $h_j := \max (h, - j)$, by the previous theorem $\MA (\hat h_j)$ is carried by the set $\{\hat h_j = h_j\}$. Hence for any $j$,  

 $$
 \int_X \min(h_j-\hat{h}_j,  1) \MA(\hat h_j)= 0.
$$
Since $h_j = h$ off the set $\{ h < - j\} \subset \{\psi < - j\}$ and $\capo (\{\psi < - j\}) \to 0$ as $j \to + \infty$, it follows that $h_j \to h$ in capacity in $X$. Hence $\phi_j:=\min(h_j-\hat{h}_j,1)$ converges to $\phi:=\min(h-\hat{h},1)$ in capacity.  Lemma \ref{lem: convergence in cap generalized} insures
\[
\lim_{j\to +\infty}\int_X |\phi_j-\phi| \MA(\hat{h}_j)=0. 
\]
On the other hand since $\phi$ is bounded and lower semi-continuous on $X$ it follows from convergence property of the complex Monge-Amp\`ere operator that 
\[
\int_X \phi \MA(\hat{h}) \leq  \liminf_{j\to +\infty} \int_X\phi\MA(\hat{h}_j). 
\]
As $\phi\geq 0$ we thus get $\int_X \phi \MA(\hat{h})=0$, finishing the proof. 
\end{proof}

\begin{lem} 	\label{lem: convergence in cap generalized}
	Fix $0\geq \varphi\in \Ec(X,\theta)$  and let $(f_j)$ be a sequence of positive uniformly bounded measurable functions on $X$ which converges in capacity to $0$. Then 
	\[
	\lim_{j\to +\infty} \sup\left\{ \int_X f_j \MA(\psi) \, :\, \psi\in \psh(X,\theta) , \varphi\leq \psi \leq 0 \right\} =0. 
	\]
\end{lem}

\begin{proof}
	Fix $\psi\in \Ec(X,\theta)$ such that $\varphi\leq \psi \leq 0$.  Since $f_j$ is uniformly bounded, we have $f_j\leq C$ for any $j$. Now for fixed $\vep>0$ we have 
	\[
	\int_X f_j \MA(\psi) \leq  C\int_{\{f_j\geq \vep \}}  \MA(\psi) + \vep  \vol(\theta) \leq C {\rm Cap}_{\varphi,V_{\theta}}(f_j\geq \vep) + \vep \vol(\theta).
	\]
	It follows from Theorem \ref{thm: DNL big} that ${\rm Cap}_{\varphi,V_{\theta}}(f_j\geq \vep)$ converges to $0$ as $j\to +\infty$. The conclusion follows by letting $\vep \to 0$.
\end{proof}

\subsection{Envelopes with respect to a measure}

Let $\mu$ be a positive measure on $X$ which does not charge pluripolar sets. 

\begin{defi} \label{def:modified}
The $(\theta,\mu)$-envelope of a  measurable function $u:X \rightarrow \R$ is defined by 
\[
P_{\theta,\mu}(u) := \left(\sup \{\varphi \in \psh(X,\theta) \, :\, \varphi\leq u\, \, \mu-a.e.\}\right)^*.
\] 
\end{defi}

This notion generalizes the one introduced in Definition \ref{def:usual}.

\begin{pro} \label{pro: modified envelope is well-defined}	
Assume that $u$ is  bounded from below and there is  $b>0$ such that $\mu(u<b)>0$. 
Then $P_{\theta,\mu}(u)$ is a well-defined $\theta$-psh function with minimal singularities. Moreover, $P_{\theta,\mu}(u)\leq u$ holds $\mu$-almost everywhere. 

If $(h_j)$ is a decreasing sequence of bounded measurable functions   which converge pointwise to $h$ in $X$,
 then $P_{\theta,\mu}(h_j)$ decreases to $P_{\theta,\mu}( h)$ in $X$.
\end{pro}

\begin{proof}
	We first prove that $P_{\theta,\mu}(u)$ is bounded from above. Indeed, fix an arbitrary $\theta$-psh function $\varphi$ such that 
	$\varphi\leq u$ $\mu$-a.e. Set 
	$$ 
	K:=\{x\in X\, :\, \varphi(x)<b\},
	$$ 
	so that $\mu(K)>0$.  
 Then $\varphi-b \leq V_{\theta,K}^*$, where $V_{\theta,K}^*$ is the global extremal $\theta$-psh function of  $K$. It follows from Theorem \ref{thm: DNL big} that $\captheta(\varphi<b)\geq c>0$, where $c$ does not depend on $\varphi$. 
 Indeed, since $\mu$ is non pluripolar we can find  $\psi\in \Ec(X,\theta)$ with $\sup_X \psi=-1$ such that $\MA(\psi)=\mu$.
	 Hence by definition of the capacity,
\[
\mrm{Cap}_{\psi,V_{\theta}}(\varphi<b)\geq \mu(\varphi<b)\geq \mu(u<b)=:  c_1>0.
\]  
By Theorem \ref{thm: DNL big} we know that $F(\captheta (E)) \geq \mrm{Cap}_{\psi,V_{\theta}}(E)$ for every Borel subset $E$. The function $F$ is  continuous and increasing, hence  $\captheta(\varphi<b)\geq F^{-1} (c_1) =: c_2>0$.  

Since the set $K$ has capacity $\geq c$, it follows from \cite{GZ05,BEGZ10} that $\sup_X \varphi\leq C$, where $C$ depends only on $c$. 
	
	Now, since the family defining $P_{\theta,\mu}(u)$ is uniformly bounded from above, the sup envelope is well defined as a $\theta$-psh function with minimal singularities. It follows from Choquet's lemma  that $P_{\theta,\mu}(u)\leq u$ holds $\mu$-almost everywhere on $X$. 

\smallskip

 	The proof of the last assertion is straightforward.
 \end{proof}

\begin{pro}	\label{prop: modified envelope}
If  $\mu$ is a volume form  and $u$ is bounded and quasi upper semi-continuous on $X$,
then $P_{\theta,\mu}(u)=P_{\theta}(u)$.
\end{pro}

\begin{proof}
	Fix an arbitrary function $\varphi\in \psh(X,\theta)$ such that $\varphi\leq u$ almost everywhere with respect to Lebesgue measure. Assume that $u$ is upper semi-continuous. Fix $x_0\in X$ and consider a local chart around $x_0$. Let $\rho$ be a smooth local potential of $\theta$ in this chart. The sub-mean value inequality yields
	\[
	\varphi(x_0) + \rho(x_0)  \leq \limsup_{\vep \to 0} \oint_{B(x_0,\vep)} (u(x)+\rho(x)) dV(x)\leq u(x_0)+ \rho(x_0),
	\]
	where the last inequality holds because $u$ is upper semicontinuous  and $\rho$ is continuous on $X$. 
	
	Assume now that $u$ is quasi upper semi-continuous on $X$. For each fixed $j\in \N$ there exists a compact set $K_j\subset X$ such that $\capo(X\setminus K_j)\leq 2^{-j}$ and the restriction of $u$ on $K_j$ is upper semi-continuous. We can also assume that $K_j$ is increasing in $j$. Let $u_j$ be a bounded function on $X$ which is upper semi-continuous and $u_j=u$ on $K_j$. 
	We can impose that $(u_j)$ is increasing in $j$.  It follows from the sub-mean value inequality that 
	 \begin{eqnarray*}
	 	\varphi(x_0) &\leq & \limsup_{\vep\to 0}\oint_{B(x_0,\vep)}u(x) dV(x)\\
	 	&\leq & \limsup_{\vep\to 0}\oint_{B(x_0,\vep)} (u(x)-u_j(x))dV(x) +  \limsup_{\vep\to 0}\oint_{B(x_0,\vep)} u_j(x) dV(x) \\
	 	&\leq & 2^{-j}\sup_X (u-u_j)  + u_j(x_0). 
	 \end{eqnarray*}
	 Thus for $x\in \cup_{j} K_j$ we have that $\varphi(x) \leq u(x)$. 
Since $X \setminus \cup_j K_j$ is pluripolar, it follows that $\varphi\leq u$ 
quasi everywhere on $X$, thus $\varphi\leq P_{\theta}(u)$. 
\end{proof}

In general the $(\theta,\mu)$-envelope is different from the usual one as the following example shows :

\begin{Exa} \label{ex: Berman convergence does not hold for general data}
Take a non-pluripolar set $E\subset X$ which has zero Lebesgue measure. Let $u$ be the function that takes value $-1$ on $E$ and $0$ on $X\setminus E$. Take $\mu=\omega^n$ and $\theta=\omega$. Then the $(\theta,\mu)$-envelope of $u$ is identically $0$ while its usual envelope  $P_{\omega}(u)$ is the relative extremal function of $E$ which is not identically zero because $E$ is non-pluripolar. 
\end{Exa}

We will see later on that for any constant $C>0$, the function $u$ defined in the example above cannot be a 
viscosity supersolution of the equation
\[
-(\omega +dd^c u)^n + C\omega^n=0.
\]

\subsection{Approximation of envelopes and proof of \ref{thmB}}\label{sect: approximation}

Fix a probability measure $\mu$ on $X$ which does not charge pluripolar sets. 
For each $j\in \N^*$ it follows from \cite{BEGZ10,BBGZ13} that there exists a unique $\theta$-psh function $\varphi_j$ with minimal singularities such that 
 \begin{equation}
 \label{eq: Berman convergence}
 \MA(\varphi_j)=e^{j(\varphi_j-u)} \mu. 
 \end{equation}
If $\mu$ is a smooth volume form and $u$ is smooth on $X$, Berman proved in \cite{Ber13} that $(\varphi_j)$ converges in energy toward the Monge-Amp\`ere envelope $P_{\theta}(u)$ \footnote{The convergence result in \cite{Ber13} has been recently  generalized to measures satisfying the Bernstein-Markov condition}. 
 The purpose of this section is to relax the regularity assumption on 
$\mu$ and $u$. We first observe the following

\begin{Lem}\label{lem: lower bound}
 Let $\phi$ be the unique function in $\Ec^1(X,\theta)$ such that $\MA_{\theta}(\phi)=e^{\phi}\mu$. Then 
 \begin{equation}\label{eq: lower bound for Berman convergence}
 \varphi_j\geq \left(1-\frac{1}{j}\right) P_{\theta,\mu}(u) + \frac{\phi}{j} + \frac{1}{j}(-n\log j +\inf_X u).
 \end{equation}
\end{Lem}

\begin{proof} 
Denote by $\psi_j$ the right-hand side of \eqref{eq: lower bound for Berman convergence} and note that $\psi_j\in \Ec^1(X,\theta)$. 
Using the fact that  $P_{\theta,\mu}(u)\leq u$ holds $\mu$-a.e. on $X$, one can  check that $\psi_j$ is a pluripotential subsolution of \eqref{eq: Berman convergence}. It thus follows from the pluripotential comparison principle that $\varphi_j\geq \psi_j$. 
\end{proof}

\begin{thm} \label{thm Berman generalized convergence}
Assume that $\mu$ is a non-pluripolar positive measure and $u$ is a bounded Borel measurable function. Then the sequence $(\varphi_j)$ converges in energy to the envelope $P_{\theta,\mu}(u)$. 
\end{thm}

\begin{proof}
In view of Lemma \ref{lem: capacity implies energy}, Lemma \ref{lem: lower bound} and Lemma \ref{lem: easy convergence in cap} it suffices to prove 
that $(\f_j)$ converges to $P_{\theta,\mu}(u)$ in $L^1$.

We claim that the sequence $(\sup_X \varphi_j)$ is  bounded. Indeed, assume this is not the case. After extracting and relabelling we can assume that $\sup_X \varphi_j \nearrow +\infty$. The sequence  $\psi_j:= \varphi_j-\sup_X \varphi_j$ is contained in a compact set of $L^1(X,\omega^n)$.  We can thus extract a subsequence, still denoted by $(\psi_j)$ that converges to $\psi\in \psh(X,\theta)$ in $L^1(X,\omega^n)$. The set 
\[
P:= \{x\in X : \sup_j\varphi_j(x) <+\infty\}
\]
is contained in $\{\psi=-\infty\}$, hence it is pluripolar.  

By assumption on $u$ there exists $s>0$ such that  $u\leq s$ on $X$. Consider 
\[
A_j:= \{x\in X : \varphi_j > 2s\}. 
\] 
As $u\leq s$ in $X$, using \eqref{eq: Berman convergence} we obtain $\mu(A_j) \leq e^{-js}\vol(\{\theta\})$. Thus for $j$ large enough, 
\[
\mu\left(\bigcup_{k\geq j} A_k\right) \leq \frac{e^{-sj}\vol(\{\theta\})}{1-e^{-s}}<\mu(X).
\]

Now the complement of $\bigcup_{k\geq j} A_k$ in $X$ is contained in $P$, a pluripolar set which is negligible with respect to $\mu$, a contradiction. Thus the claim is proved, i.e. $\sup_X \varphi_j$ is bounded.

It follows now from compactness properties of $\theta$-psh functions (see \cite{GZ05})
that the sequence $(\varphi_j)$ is relatively compact in $L^1$. It suffices to prove that any cluster point of this sequence coincides with $P_{\theta,\mu}(u)$. 
Let $\varphi$ be such a cluster point. Extracting   and relabelling we can assume that $\varphi_j$ converges in $L^1$ to $\varphi$.  It follows from Lemma  \ref{lem: lower bound}  that $\varphi\geq P_{\theta,\mu}(u)$. Consider 
\[
\tilde{\varphi}_j:= \left(\sup_{k\geq j} \varphi_j \right)^*. 
\]
Then $\tilde{\varphi}_j$ decreases pointwise to $\varphi$. Fix $\vep>0$ and set 
\[
U_j:= \{x\in X : \varphi_j(x)>u(x)+\vep\}\ ; \ \tilde{U}_j:=  \{x\in X : \tilde{\varphi}_j(x)>u(x)+\vep\}.
\]
As negligible sets are also pluripolar and $\mu$ does not charge these sets we get
\[
\mu(\tilde{U}_j) \leq \sum_{k\geq j} \mu(U_k) \leq \frac{\vol(\{\theta\}) e^{-j\vep}}{1-e^{-\vep}}.
\]
Thus $\mu(U)=0$ where $U:=\cap_{j\in \N^*} \tilde{U}_j\supset \{\varphi>u+\vep\}$. 
Letting $\vep\to 0$ we obtain $\mu(\varphi>u)=0$ hence $\varphi \leq P_{\theta,\mu}(u)$, finishing the proof. 
\end{proof}

\section{Viscosity vs pluripotential supersolutions} \label{sec:viscosity}

Let $(X,\omega)$ be a compact K\"ahler manifold of dimension $n$ and fix $\theta$ a  smooth closed $(1,1)$-form on $X$ which represents a big cohomology class. 

\subsection{Background on viscosity solutions}
Let $f:X \rightarrow \R^+$ be a continuous function.

\begin{defi}
Let $u: X\rightarrow \mathbb{R}\cup \{-\infty\}$ be an upper semicontinuous function.
An upper test function  for $u$ at $x_0$ is a function $q:V_{x_0} \rightarrow \R$ defined 
in a neighborhood $V_{x_0}$ of $x_0$ such that
$u\leq q$ in  $V_{x_0}$ with  $u(x_0)=q(x_0)$. 
\end{defi}

One can define similarly lower tests.
Upper test functions are used to define the notion of viscosity subsolutions:

\begin{defi}
We say that the inequality 
\[
(\theta +dd^c u)^n \geq e^{u} fdV 
\]
holds in the viscosity sense in an open subset $U\subset X$ if $u$ is finite on $U$ and if for any $x_0\in U$ and any $\mathcal{C}^2$ upper test function $q$ of $u$ at $x_0$
 the inequality $(\theta+dd^c q)^n\geq e^{q}fdV$ holds at $x_0$. 
 \end{defi}
 We shall equivalently say that $u$ is a viscosity subsolution of the equation
 $(\theta + dd^c u)^n = e^\f f d V$.
 
 The notion of viscosity supersolution is defined similarly with a subtle twist that we emphasize now:

\begin{defi}
We say that the inequality 
\[
(\theta +dd^c u)^n \leq e^{u} fdV 
\]
holds in the viscosity sense in an open subset $U\subset X$ if $u$ is finite on $U$ and if for any $x_0\in U$ and any $\mathcal{C}^2$ 
lower test function $q$ of $u$ at $x_0$
 the inequality 
 $$
 (\theta+dd^c q)_+^n \leq e^{q}fdV
 $$
  holds at $x_0$. 
 \end{defi}
 
 Here $\alpha_+=\alpha$ if the $(1,1)$-form $\alpha$ is semipositive and $\alpha_+=0$ otherwise. 

\begin{defi}
 A {\it viscosity solution} is a function that is both a viscosity subsolution and a viscosity supersolution.
 \end{defi}

 We refer the reader to \cite{EGZ11,Wang12,Ze13} for basic properties of viscosity sub/super-solutions to 
 degenerate complex Monge-Amp\`ere equations. We stress that viscosity subsolutions are $\theta$-psh \cite{EGZ11} and admit upper tests at almost every point by \cite{DD16}. The analogous properties for viscosity supersolutions are far less obvious.

\subsection{The global context}
 Consider the following Monge-Amp\`ere equation
 \begin{equation}\label{eq: viscosity super-solution vs pluripotential 1}
 (\theta+dd^c u)^n = e^u f\omega^n,
 \end{equation}
 where $f$ is a non-negative continuous function on $X$.
 
 \subsubsection{Envelope of viscosity supersolutions}
 
 \begin{thm}\label{thm: visco super vs pluri compact}
If $u$ is a viscosity super-solution of \eqref{eq: viscosity super-solution vs pluripotential 1} then $P_{\theta}(u)$ is a pluripotential super-solution of \eqref{eq: viscosity super-solution vs pluripotential 1}. 
 \end{thm}
 
This connection has been observed in \cite[Lemma 4.7.3]{EGZ11} when $u$ is ${\mathcal C}^2$-smooth.
The key idea of the proof given here is to approximate the Monge-Amp\`ere envelope $P_{\theta}(u)$ 
as in Theorem \ref{thm Berman generalized convergence}.
 
 \begin{proof}
For each $\beta>0$ let $\varphi_{\beta}$ be the unique $\theta$-psh function with minimal singularities such that 
 \begin{equation}\label{eq: viscosity super-solution vs pluripotential 2}
 (\theta +dd^c \varphi_{\beta})^n = e^{\beta(\varphi_{\beta}-u)} e^{\varphi_{\beta}} f\omega^n
 \end{equation}
 holds in the pluripotential sense.  The existence and uniqueness of the solution $\varphi_{\beta}$ with minimal singularities follows from the main result of \cite{BEGZ10, BBGZ13}. 
 As shown in \cite{EGZ11,EGZ16} the equation \eqref{eq: viscosity super-solution vs pluripotential 2} holds in the viscosity sense in 
$\Omega:=\Amp(\theta)$  (the ample locus of $\{\theta\}$) as well. 
 
 \smallskip
 
 {\it Step 1.
  We claim that $\varphi_{\beta}\leq u$ in $\Omega$, for all $\beta>0$}. 
 Recall that $u$ is a viscosity super-solution of \eqref{eq: viscosity super-solution vs pluripotential 1}. 
 Thus the claim would follow if we could apply the viscosity comparison principle \cite{EGZ11,EGZ16}. However, the density in the second term of \eqref{eq: viscosity super-solution vs pluripotential 2}, $e^{-\beta u} f$, is not continuous in $\Omega$,
 hence one cannot directly apply the results from \cite{EGZ11,EGZ16}.

  To prove the claim we proceed by approximation. Fix $\beta>0$ and let $(u_j)$ be an increasing sequence of continuous functions converging to $u$. 
  Such a sequence exists because $u$ is lower semicontinuous.  For each $j$ let $\varphi_{\beta,j}$ be the unique $\theta$-psh function with minimal singularities such that
  \begin{equation}\label{eq: viscosity super-solution vs pluripotential 3}
 (\theta +dd^c \varphi_{\beta,j})^n = e^{\beta(\varphi_{\beta,j}-u_j)} e^{\varphi_{\beta,j}} f\omega^n.
 \end{equation}
 As $u_j\leq u$  one can  check that $u$ is a viscosity super-solution of \eqref{eq: viscosity super-solution vs pluripotential 3}. 
  Indeed 
 $$
 (\theta+dd^c u)_+^n \leq e^{u} f \omega^n  \leq e^{\beta(u-u_j)} e^{u} f \omega^n.
 $$
 Moreover, the density function $e^{-\beta u_j} f$ is continuous on $X$. Hence by the viscosity comparison principle \cite{EGZ16} it follows that $\varphi_{\beta,j}\leq u$ in $\Omega$. For $j>k$, as $u_j\geq u_k$ we have 
 \[
 (\theta +dd^c \varphi_{\beta,k})^n = e^{\beta(\varphi_{\beta,k}-u_k)} e^{\varphi_{\beta,k}} f\omega^n \geq e^{\beta(\varphi_{\beta,k}-u_j)} e^{\varphi_{\beta,k}} f\omega^n 
 \]
 in the pluripotential sense. In other words, $\varphi_{\beta,k}$ is a pluripotential subsolution of \eqref{eq: viscosity super-solution vs pluripotential 3}.
It follows therefore from the pluripotential comparison principle (Proposition \ref{prop: comparison principle exponential}) that  $j \mapsto \varphi_{\beta,j}$ is increasing. 
  The increasing limit   is a  $\theta$-psh function with minimal singularities which solves the equation  \eqref{eq: viscosity super-solution vs pluripotential 2}
  (recall that the Monge-Amp\`ere operator is continuous along monotonous sequences). By uniqueness it follows that this limit is $\varphi_{\beta}$. 
 This proves the claim since  $\varphi_{\beta,j}\leq u$ in $\Omega$.
 
 \smallskip
 
{\it Step 2. We now claim that $\f_{\beta}$ increases  towards $P(u)$, as $\beta \rightarrow +\infty$.}
Since $\f_{\beta} \leq u$, we observe that $\f_{\gamma}$ is a subsolution to 
$(\ref{eq: viscosity super-solution vs pluripotential 2})_{\beta}$ if $\gamma \leq \beta$, hence
 $\beta \mapsto \varphi_{\beta}$ is increasing.
 It  converges   almost everywhere to 
some function $\varphi\in \psh(X,\theta)$ with minimal singularities such that $\f \leq u$ in $\Omega$.
Since $X \setminus \Omega$ is pluripolar we infer $\varphi\leq P_{\theta}(u)$ 
on $X$.
 
We now show that $\varphi \geq P_{\theta}(u)$.  Using the domination principle 
it suffices to prove that  $\MA(\varphi)$ vanishes in $\{\varphi <P_{\theta}(u)\}$.
Fix $\delta>0$. Using \eqref{eq: viscosity super-solution vs pluripotential 2} and observing that
$$
\{\varphi<P(u) -\delta\} \subset \{\varphi_{\beta} <u -\delta\},
$$
we obtain
 \[
\int_{\{\varphi<P(u) -\delta\}}\MA(\varphi_{\beta})   \leq e^{-\beta \delta} \int_X e^{P(u)}fdV, \ \forall \beta>1.
\] 
Since $MA(\varphi_{\beta})$ weakly converges  to $MA(\varphi)$ as $\beta\to +\infty$,
it follows from Lemma \ref{lem: plurifine convergence} that
 \[
\int_{\{\varphi<P(u) -\delta\}} \MA(\varphi) \leq \liminf_{\beta\to +\infty} \int_{\{\varphi<P(u) -\delta\}} \MA(\varphi_{\beta}) =0.
\] 
Letting $\delta \to 0$ insures that $\MA(\varphi)$   vanishes in $\{\varphi<P(u)\}$. 

\smallskip

{\it Conclusion.} Recall  that 
\[
(\theta + dd^c \varphi_{\beta})^n \leq e^{\varphi_{\beta}}f\omega^n
\]
in the pluripotential sense. Since $\varphi_{\beta}$ increases to $P_{\theta}(u)$,   the continuity of the Monge-Amp\`ere operator
along monotonous sequences insures that
\[
(\theta + dd^c P(u))^n \leq e^{P(u)}f\omega^n
\]
in the pluripotential sense, as desired.
 \end{proof}

 \begin{rem}
 Recall that a partial converse to this implication has been given in \cite[Lemma 4.7]{EGZ11}:
 if $\p$ is a bounded pluripotential supersolution, then its lower semi-continuous regularization $\p_*$ is
 a viscosity super-solution.
 \end{rem}

 \subsubsection{More general RHS}
 We consider the following generalization. 

\begin{thm}
	\label{thm: supersolution general right hand side} 
	Assume that $F: X \times \R \rightarrow \R^+$ is a continuous function which is non-decreasing in $t \in \R$. Let $u$ be a viscosity supersolution of the equation 
	\begin{equation}
		\label{eq: viscosity supersolution general right hand side}
		(\theta +dd^c u)^n = F(x,u) \omega^n.
	\end{equation}
	Then the envelope $P_{\theta}(u)$ is a pluripotential supersolution of \eqref{eq: viscosity supersolution general right hand side}.
\end{thm} 

\begin{proof}
	The proof, similar to that of Theorem \ref{thm: visco super vs pluri compact}, is left to the reader. 
\end{proof}

 \subsubsection{Continuity of envelopes}

Let  $\theta$ be a semi-positive and big form on $X$.
We want to investigate  conditions under which the envelope $P_{\theta}(u)$ is continuous in $X$ when $h$ is continuous in $X$. 
We say that 
{\it $(X,\theta)$ satisfies the approximation property $(AP)$} if any $\f \in \psh (X,\theta)$ can be approximated by a decreasing sequence of continuous $\theta$-psh functions.
 
 \begin{thm} 
 The following properties are equivalent
 
 $(i)$  For any $h \in C^0 (X)$, its envelope $P_{\theta}(h)$ is continuous on $X$;
 
 $(ii)$ $(X,\theta)$ satisfies the approximation property $(AP)$;
 
 $(iii)$ For any density $0 \leq f \in L^{\infty} (X,\R)$ with $\int_X f \theta^n = \int_X \theta^n$, the unique solution to the complex Monge-Amp\`ere equation
 \begin{equation} \label{eq:MA}
  (\theta + dd^c \f)^n = f d V,   \, \, \max_X \f = 0,
 \end{equation}
 is continuous on $X$. 

 \end{thm}
 
 The approximation property has been introduced in \cite{EGZ09}
  where it is proved that $(ii) \Longrightarrow (iii)$  holds.
  $(AP)$ is known to hold when $\{ \theta \}$ is a K\"ahler class \cite{Dem92,BK07,Ber13}, or when
  it is a Hodge class on a singular variety \cite{CGZ13}.
 
 \begin{proof} 
$(i) \Longrightarrow (ii)$. Assume that $\f \in \psh(X,\theta)$. Let $(h_j)$ be a sequence of smooth functions decreasing to $\f$ on $X$. Then using $(i)$ we conclude that $(P(h_j))$ is a decreasing sequence of continuous $\theta$-psh functions in $X$ that converges to $\f$.

$(ii) \Longrightarrow (iii)$. This property follows from \cite{EGZ09}.

 $(iii) \Longrightarrow (i)$. Assume first that $h$ is smooth on $X$. Then $(\theta + dd^c h)^n_+ = f d V$, where $f$ is a continuous function on $X$. Then $h$ is a viscosity supersolution to the complex Monge-Amp\`ere equation 
 $$
 (\theta + dd^c \phi)^n = e^{\phi - h} (\theta + dd^c h)_+^n = e^\phi e^{- h} f d V,
 $$
 where $f e^{- h}$ is a continuous density on $X$.
 
 By \ref{thmA}, $P(h)$ is then a pluripotential supersolution of the same equation. Therefore $(\theta + dd^c P(h))^n \leq f d V$, in the weak sense, hence there exists a function $g \in L^{\infty} (X)$ such that $(\theta + dd^c P(h)) = g d V$ weakly on $X$. Hence by $(iii)$, $P(h)$ is continuous on $X$. 
 
 The general case follows by approximation. Let $h \in C^0 (X,\R)$ then approximate $h$ by a decreasing sequence of smooth functions $h_j$ in $X$. Since $|P(h_j)-P(h)|\leq \sup_X |h_j-h|$, it follows that  $P(h_j)$ converges  to $P(h)$ uniformly on $X$, hence $P(h)$ is continuous on $X$.
 \end{proof}

 \subsection{The local context}
 
 Let $D \Subset \C^n$ be a bounded hyperconvex domain. By definition $D$ admits a continuous negative plurisubharmonic exhaution $\rho$.  The domain $D$ is said to be strictly pseudoconvex if the exhaution function $\rho$ can be choosen strictly plurisubharmonic in $D$.  Let $0\leq f$ be a continuous function in $D$ and $d V$ be the Euclidean volume form on $D$. 
 
 For a bounded function $u$ in $D$, the upper envelope $P_{D}(u)$ of $u$  in $D$ is defined by 
 \[
 P_{D}(u) := \left( \sup\{\varphi \in \psh(\Omega) : \varphi \leq u\} \right)^*. 
 \]
 We will also need to consider the following envelope which takes care of the boundary values:
\[
P_{\bar{D}} (u) := (\sup\{\varphi \in \mathrm{PSH}(D)  : \varphi^* \leq u\ \textrm{on } \bar{D}\})^*,
\]
where $\varphi^*$ is the upper semicontinuous extension of $\varphi$ to $\bar{D}$ defined by  
\[
\varphi^*(\xi):= \lim_{r\to 0^+} \sup_{B(\xi,r)\cap D} \varphi, \ \xi \in \partial D. 
\] 
Note that the extension of $\varphi$ to $\bar{D}$ is upper semicontinuous on $\bar{D}$. 

If $u$ is continuous on $\bar{D}$ then $P_{D}(u)=P_{\bar{D}}(u)$ on $D$. This does not hold in general as the example  in Remark \ref{rem:twoenvelopes} below shows.
\medskip

We now state and prove the local version of \ref{thmA} :
 \begin{thm} \label{thm:supersolloc}
 	Let  $D \Subset \C^n$ be a bounded pseudoconvex domain. Assume that a bounded lower semi-continuous function $u$ is a viscosity super-solution of the equation
 	\begin{equation}
 		\label{eq: viscosity supersolution local 1}
 		(dd^c u)^n =  fdV,
 	\end{equation}
 	in $D$. Then $P_{D}(u)$ is a pluripotential super-solution of \eqref{eq: viscosity supersolution local 1} in $D$. 
 \end{thm}
 
 \begin{proof}
 	We first assume that $D$ is strictly pseudoconvex and $u$ is continuous in $\bar D$ and prove that $P_{D}(u)$ is a pluripotential super-solution of \eqref{eq: viscosity supersolution local 1}. For each $\beta>0$, let $\varphi_{\beta}$ be the unique function in $\psh(D)\cap C^0 (\bar D)$ such that
 	\begin{equation}
 		(dd^c \varphi_{\beta})^n = e^{\beta(\varphi_{\beta}-u)} f dV \ \ {\rm in }\ D
 	\end{equation}
 	with boundary values $u$ i.e. $\varphi_{\beta} = u$ in $\partial D$ \cite{BT76}. \\
 	
 	Using the local viscosity comparison principle \cite{EGZ11},  we deduce that $\varphi_{\beta} \leq u$ in $\bar D$ for any $\beta > 0$.
 	Arguing as in Step 2 of the proof of Theorem~\ref{thm: visco super vs pluri compact}, we can prove that   {\it 	 $\varphi_{\beta}$ increases to $P_D (u)$ a.e. in $D$ as $\beta$ increases to $+ \infty.$}	
 	Therefore by letting $\beta \to +\infty$ in the above equation, we conclude that $(dd^c P_D (u))^n \leq f dV$ in $D$. 
 	
 	\smallskip
 	
 	Now assume that $D \Subset \C^n$ is a pseudoconvex domain and $u$ is continuous in $D$. Then the result follows by taking an exhaustive sequence of strictly pseudoconvex domains  $D_j\Subset D$ (in view of the lemma below) and applying what have been done above.  
 	
 	For the general case when $u$ is merely lower-semi continuous, we approximate $u$ by inf-convolution i.e. we consider 
 	$$
 	u^j (z) := \inf \{ u (\zeta) + j \vert z - \zeta\vert^2 \}, \ \ z \in D^{j}, j \in \N^*,
 	$$
 	where $(D^j)$ is an exhaustive sequence of pseudoconvex domains converging to $D$.
 Then we know that $(u^j)$ is an increasing sequence of continuous functions converging to $u$ in $D$ and for each $j$, the function $u^jj$ is a supersolution of 
 $$
 (dd^c v)^n  = f_j d V, \, \, \, {\rm in }\ D^{j},
 $$
 where  $f_j$ is continuous in $D^j$ and the sequence $(f_j)$ decreases to $f$ in $D$ (see \cite{CIL92,CC95}).
 
 Fix any pseudoconvex domain $B \Subset D$. The previous result insures that for $j > 1$ large enough so that $B \subset D^j$, the  function  $P_B (u_j)$ satisfies $(dd^c P_B (u^j))^n \leq f_j$ in the pluripotentiel sense in $B$.
 Aplying Lemma \ref{lem: increasing sequence} below we obtain at the limit that  the differentiel inequality $(dd^c P_{\bar B} (u))^n \leq f$ holds in the pluripotential sense on $B$.
 
 Again taking an exhaustive  sequence $(B_j)$ of pseudoconvex domains converging to $D$ and applying  Lemma \ref{lem:exhaustion} below we obtain the required result.
 \end{proof}

\begin{lem}\label{lem: increasing sequence}
	Assume that $(u_j)$ is an increasing sequence of lower semicontinuous functions on $\bar{D}$ which converges pointwise to $u$. Then $P_{\bar{D}}(u_j)$  increases almost everywhere to $P_{\bar{D}}(u)$.  As a consequence, if $u_j$ is continuous on $\bar{D}$ for all $j$, then $P_{D}(u_j)$ increases almost everywhere to $P_{\bar{D}}(u)$.  
\end{lem}

\begin{proof}
	 Let $\varphi$ be an arbitrary psh function in $D$ such that $\varphi\leq u$ on $\bar{D}$.   
	
	Fix $\varepsilon>0$. We will show that $u_j\geq \varphi-\varepsilon$ on $\bar{D}$, for $j$ large enough. Assume that this was not the case. Then we can find a sequence $(x_j)\subset \bar{D}$ such that $u_j(x_j) <\varphi(x_j)-\varepsilon$. For $k\in \mathbb{N}$ fixed and $j>k$ we have 
	$$
	u_k(x_j)\leq u_j(x_j) \leq \varphi(x_j)-\vep.
	$$ 	
 We can assume that $x_j\to x\in \bar{D}$. Since $u_k$ is lsc and $\varphi$ is usc on $\bar{D}$ it follows that $u_k(x) \leq \varphi(x) -\varepsilon$. Since this is true for any $k$ we deduce that $u(x)\leq \varphi(x)-\varepsilon \leq  u(x)-\varepsilon$, a contradiction. 
	
	We thus have that $P_{\bar{D}}(u_j) \geq \varphi - \varepsilon$, for $j$ large enough. 
	By letting $j\to +\infty$ and then $\varepsilon\to 0$ we conclude that $(\lim P_{\bar{D}}(u_j))^* \geq \varphi$, ultimately giving $(\lim P_{\bar{D}}(u_j))^* \geq P_{\bar{D}}(u)$. The reverse inequality is obvious.  \end{proof}

\begin{lem}
	\label{lem:exhaustion}
	Let $(D_j)$ be an increasing sequence of relatively compact bounded domains in $D$ such that $\cup D_j=D$. Let $u$ be a lower semicontinuous function in $D$. Then $P_{\bar{D}_j}(u)$ decreases pontwise to $P_{D}(u)$. 
\end{lem}
\begin{proof}
	Set $\varphi_j=P_{\bar{D}_j}(u)$ and note that this is a psh function in $D_j$. Clearly, $\varphi_j$ is decreasing in $j$. The decreasing limit $\lim_{j} \varphi_j$ is psh in any $D_k$. As $(D_k)$ is an exhaustive sequence of $D$, these limits define a psh function $\varphi$ in $D$. 
	
	We need to prove that $\varphi=P_{D}(u)$. Indeed, if $v$ is a psh function in $D$ such that $v\leq u$ in $D$ then $v$ is also a candidate defining $P_{\bar{D}_j}(u)$, thus $\varphi_j \geq v$. We then get $\varphi \geq P_D(u)$. On the other hand, $\varphi$ is psh in $D$ and $\varphi\leq u$ quasi everywhere in $D$, i.e. there exists $E\subset D$ a pluripolar set such that $\varphi\leq u$ in $D\setminus E$. By Josefson theorem there exists a negative psh function $\psi$ in $D$ such that $\psi\not \equiv -\infty$ and $\psi=-\infty$ on $E$. For any $\delta>0$ observe that $\varphi +\delta \psi \leq u$ in $D$. Therefore $\varphi+\delta \psi \leq P_D(u)$ for all $\delta>0$. Letting $\delta\to 0^+$ we get $\varphi\leq P_{D}(u)$ quasi everywhere hence everywhere in $D$. 
 \end{proof}

\begin{rem}\label{rem:twoenvelopes}
	We stress that the envelopes $P_D(u)$ and $P_{\bar{D}}(u)$ are in general different if $u$ is not continuous near the boundary. Indeed, take 
	the function $u$ defined on $\bar D$ by $u = 0$ in $D$ and $- 1$ in $\partial D$. Then $u$ is lower semi-continuous in $\bar D$,
	$P_D (u) = 0$ in $D$ while $P_{\bar D} (u) = - 1$ in $D$.
\end{rem}

\section{The minimum principle} \label{sec:resolutionMA}

\subsection{The minimum principle}

The following property is inspired by the fact that the minimum of two viscosity supersolutions is again a viscosity supersolution
(see Section \ref{sec:viscosity}).

\begin{lem} \label{lem: partition inequality}
Let $u,v\in \Ec(X,\theta)$ and set $\varphi:=P_{\theta}(\min(u,v))$.  Then $\varphi\in \Ec(X,\theta)$ and 
\[
\MA(\varphi) \leq {\bf 1}_{\{\varphi=u\}} \MA(u) + {\bf 1}_{\{\varphi=v\}} \MA(v). 
\]
\end{lem}

When $\theta$ is semipositive and $u,v$ have bounded laplacian the result follows from Darvas' work \cite{Dar14,Dar15} which uses a strong regularity result on the Monge-Amp\`ere envelope. Our proof is inspired by the convergence  method of Berman \cite{Ber13}.

\begin{proof}
 The fact that $\varphi$ belongs to $\Ec(X,\theta)$ follows from \cite{DDL16}.  Without loss of generality we can assume that $u,v$ have minimal singularities. 
 
For each $j\in \N^*$ it follows from Lemma \ref{lem: solve MA eq} below
that there exists a unique $\varphi_j\in \psh(X,\theta)$ 
with minimal singularities such that  
\[
\MA(\varphi_j) = e^{j(\varphi_j-u)} \MA(u)  + e^{j(\varphi_j-v)}\MA(v). 
\]
As both $u$ and $v$ are pluripotential supersolutions of the above equation it follows that $\varphi_j\leq \varphi$. 
By the pluripotential comparison principle 
we also have that $\varphi_j$ increases almost everywhere to $\varphi_{\infty}\in \Ec(X,\theta)$.

We first prove that $\varphi_{\infty}=\varphi$. For each fixed $\vep>0$ one has 
\[
\int_{\{\varphi_{\infty}<\varphi-\vep\}} \MA(\varphi_j)\leq \int_{\{\varphi_j<\varphi-\vep\}} \MA(\varphi_j) \leq 2\vol(\theta)e^{-j\vep}\to 0. 
\]
Letting $j\to +\infty$ we obtain that $\MA(\varphi_{\infty})$ vanishes in $\{\varphi_{\infty}<\varphi\}$. 
Applying the domination principle, Proposition \ref{prop: domination principle}, 
gives $\varphi_{\infty}=\varphi$. 

Now we prove the inequality in the statement of the lemma. For each fixed $A>0$ and $j>A$, since $\varphi_j\leq \min(u,v)$ we have
\[
\MA(\varphi_j) \leq e^{A(\varphi_j-u)} \MA(u) + e^{A(\varphi_j-v)} \MA(v) 
\]
Since $\MA(\varphi_j)$ converges to $\MA(\varphi)$, by dominated convergence theorem we obtain 
\[
\MA(\varphi) \leq e^{A(\varphi-u)} \MA(u) + e^{A(\varphi-v)} \MA(v).
\]
Now, letting $A\to +\infty$ we obtain the result. 

\end{proof}

\begin{lem}
\label{lem: solve MA eq}
Assume that $u,v$ are $\theta$-psh functions with minimal singularities and fix $\beta>0$. Then  there exists a unique  $\theta$-psh function $\varphi$ with minimal singularities such that 
\[
\MA(\varphi)=e^{\beta(\varphi-u)} \MA(u)+ e^{\beta(\varphi-v)}\MA(v). 
\]	
\end{lem}

\begin{proof}
The uniqueness follows from the comparison principle. 

	To prove existence, without loss of generality we can assume that $\beta=1$. For each $j\in \N$ set $u_j=\max(u,-j), v_j=\max(v,-j)$.  As $e^{-u_j}+ e^{-v_j}$ is bounded on $X$ it follows from \cite{BBGZ13} that there exists a unique  $\varphi_j\in \Ec^1(X,\theta)$ such that 
	\[
	\MA(\varphi_j) =e^{\varphi_j-u_j} \MA(u) +e^{\varphi_j-v_j}\MA(v). 
	\]
	By the comparison principle we know that $\varphi_j$ is decreasing in $j$ and $\varphi_j\geq \frac{u+v}{2} -C$, for some constant $C>0$ independent of $j$. We can thus pass to the limit $j\to +\infty$ and obtain the result. 
\end{proof}

\begin{lem}\label{lem: Darvas 1}
Let $u_j,u\in \Ec^p(X,\theta)$, $p>0$, be such that $(u_j)$ converges to $u$ in energy $I_p$. Then there exists a subsequence still denoted by $(u_j)$ such that 
\[
\varphi_j:= P_{\theta}\left(\inf_{k\geq j} u_k\right) \nearrow u. 
\]
\end{lem}

The proof is an adapatation of an argument due to Darvas \cite{Dar14,Dar15}.

\begin{proof}
After extracting a subsequence we can assume that 
\[
\int_X |u_j-u|^p  \MA_{\theta}(u_j) \leq 2^{-j}, \ \forall j. 
\]
For each $j<k$ set $\varphi_{j}^k :=P_{\theta}\left(\inf_{j\leq \ell\leq k} u_{\ell}\right)$. It follows from \cite{DDL16} that $\varphi_{j}^k$ belongs to $\Ec^p(X,\theta)$. It follows from Lemma \ref{lem: partition inequality} that 
\[
\int_X |u-\varphi_{j}^k|^p \MA(\varphi_{j}^k) \leq \sum_{\ell =j}^k \int_X |u-u_{\ell}|^p \MA(u_{\ell}) \leq 2^{-j+1}. 
\]
It then follows from Lemma \ref{lem: BEGZ prop 2.10}  that the decreasing limit $\varphi_j:=\lim_{k\to +\infty}(\varphi_{j}^k)$ belongs to $\Ec^p(X,\theta)$. Moreover the continuity of the Monge-Amp\`ere operator (see \cite{BEGZ10}) gives 
\[
\int_X |u-\varphi_{j}|^p \MA(\varphi_{j}) \leq 2^{-j+1}. 
\]
Hence the increasing limit $\varphi:=\lim_{j\to +\infty}\varphi_j\leq u$ satisfies $\int_{\{\varphi<u\}}\MA(\varphi)=0$. The domination principle then reveals that $u=\varphi$. 
\end{proof}

\subsection{Solving complex Monge-Amp\`ere equations}

Let $(X,\omega)$ be a compact K\"ahler manifold and $\mu$ a positive non-pluripolar Radon measure in some open subset
$\Omega$  of $X$. Here we allow $\mu$ to have infinite total mass (i.e. $\mu(\Omega)$ may be $+\infty$). 

 Let $\theta$ be a smooth closed $(1,1)$-form on $X$ which represents a big class. We assume that there exists  
  $u_0 \in {\mathcal E}(X,\theta)$,  such that 
 $$
 (\theta+dd^c u_0)^n \geq e^{u_0} \mu.
 $$
 In particular $e^{u_0} \mu$ extends as a Radon measure in all of $X$.
 
 We would like to solve the complex Monge-Amp\`ere equation 
 \begin{equation} \label{eq: min of supersolution}
 (\theta+dd^c \f)^n =e^{\f} \mu,
 \end{equation}
  by considering the lower envelope of supersolutions.
 We first note the following simple consequence of Lemma \ref{lem: partition inequality} which is basic to what follows:
 
 \begin{pro} \label{pro:min}
 Assume $u,v \in \Ec(X,\theta)$ are both supersolutions of \eqref{eq: min of supersolution}, i.e. $\MA(u) \leq e^u \mu$, $\MA(v) \leq e^{v} \mu$. Then
 $P(\min(u,v))$ is also a supersolution,  
 $$
 \MA(P(\min(u,v)) \leq e^{P(\min(u,v))} \mu.
 $$
 \end{pro}
 
 \begin{proof}
 Observe first that since $\mu$ is a Radon measure in
 $\Omega$, $\mu(u=v+r)=0$ for almost every $r\in \R$. Fix such an $r$ and set $\varphi_r=P_{\theta}(\min(u, v+r))$. It follows from Lemma \ref{lem: partition inequality} that  
 \begin{eqnarray*}
 	 \MA(\varphi_r) &\leq & {\rm 1}_{\{\varphi_r=u\}} \MA(u) + {\rm 1}_{\{\varphi_r=v+r\}}\MA(v)\\
&\leq & {\rm 1}_{\{\varphi_r=u\}} e^u \mu  + {\rm 1}_{\{\varphi_r=v+r\}}e^v \mu \\
& \leq & e^{\varphi_r+|r|}\mu.
 \end{eqnarray*}
 Proposition \ref{prop: basic property of envelope} insures that $\varphi_r$ decreases pointwise to $P_{\theta}(\min(u,v))$. The latter belongs to $\Ec(X,\theta)$ as follows from \cite{DDL16}. We conclude by letting $r \to 0$. 
 \end{proof}

 This result guarantees that the families of pluripotential super-solutions is stable under the 
 operation $P(\min(\cdot,\cdot))$. One can thus hope and solve the equation by taking
 the infimum of supersolutions; this is the contents of  \ref{thmC} from the introduction
 which we now establish :

\begin{thm} \label{thm: subsolution theorem}
   Assume there exists a 
 subsolution $u_0 \in {\mathcal E}(X,\theta)$, i.e. 
 $$
 (\theta+dd^c u_0)^n \geq e^{u_0} \mu.
 $$
 
 Then the envelope of supersolutions
 $$
 \f:=P\left( \inf \{ \p, \p \in {\mathcal E}(X,\theta) \text{ and } (\theta+dd^c \p)^n \leq e^{\p} \mu\} \right)
 $$
 is the unique pluripotential solution of $ (\theta+dd^c \f)^n = e^{\f} \mu$.
\end{thm}

This result largely generalizes the main result of Berman-Guenancia \cite[Theorem A]{BG14} : 
a projective complex algebraic variety $V$ with semi-log canonical
singularities and ample canonical bundle admits a unique K\"ahler-Einstein metric.
Constructing the latter indeed boils down to solving a complex Monge-Amp\`ere equation as above,
where 
\begin{itemize}
\item $\pi: X \rightarrow V$ is a resolution of singularities,
\item  $\theta=\pi^* \omega_V$ is the pull-back of a K\"ahler form representing $c_1(V)$,
\item $\mu=fdV$ is absolutely continuous with respect to Lebesgue
measure with a density $0 \leq f$ which is smooth in $X \setminus D=\pi^{-1}V^{reg}$,
and blows up near $D=(s_D=0)$ like $|s_D|^{-2}$.
\end{itemize}   
One easily constructs a subsolution in this case (take e.g. $-(-\log |s_D|^{-2})^{a}$ with
$0<a<1$ and $s_D$ appropriately normalized).

\begin{proof}
Let $K$ be a compact subset of $\Omega$ such that $0<\mu(K)<+\infty$ and denote by $\mu_K$ the restriction of $\mu$ on $K$, which is a positive non-pluripolar measure on $X$. It follows from \cite{BBGZ13} that there exists $\varphi_K\in \mathcal{E}(X,\theta)$ such that $\MA(\varphi_K)=e^{\varphi_K} \mu_K$. Hence $\varphi_K$ is a supersolution of \eqref{eq: min of supersolution}.  The family ${\mathcal F}$ of supersolutions is thus non-empty, and it is uniformly  bounded from below by the subsolution $u_0$, as follows from the comparison principle.

It follows from Proposition \ref{pro:min} that ${\mathcal F}$ is stable by $P \circ \min (\cdot, \cdot)$.
Coupled with an analogue of a classical lemma due to Choquet (see Lemma \ref{lem: analogue of Choquet}), this insures that
$$
 \f:=P\left( \inf \{ \p, \p \in {\mathcal E}(X,\theta) \text{ and } (\theta+dd^c \p)^n \leq e^{\p} \mu\} \right)
 $$
 is again a supersolution to the equation: this is the minimal supersolution. 
  
In order to prove that $\f$ is actually the solution it suffices to show that $\MA(\varphi)=e^{\varphi}\mu$ in any small ball $B\subset \Omega \cap \Amp(\theta)$. Fix such a ball $B$. We  construct a supersolution $\psi$ which is smaller than $\varphi$ and which solves the equation in $B$. The classical method to produce such a supersolution is to solve a local Dirichlet problem in $B$ and glue the local function with $\varphi$ on the boundary $\partial B$. This requires a subtle analysis near the boundary as the functions at hand are not continuous (they may be even unbounded).  We provide rather a global method which is simpler. 

For each $j\in \mathbb{N}$ we let $\psi_j\in \Ec(X,\theta)$ be the unique solution to 
\[
\MA(\psi_j) =  \mathrm{\bf 1}_{X\setminus B} e^{\psi_j-\max(\varphi,-j)} \MA(\varphi) + \mathrm{\bf 1}_B e^{\psi_j}\mu.
\]
The existence of such a solution follows from \cite{BBGZ13}, observing that  
$$
\nu_j=\mathrm{\bf 1}_{X\setminus B} e^{-\max(\varphi,-j)} \MA(\varphi) + \mathrm{\bf 1}_B \mu
$$
is a non pluripolar Radon measure on $X$.
Since $\MA(\varphi)\leq e^{\varphi}\mu$, one can check that $u_0$ is a subsolution of the above equation. It thus follows from the comparison principle that $\psi_j \geq u_0$ decreases to $\psi \in \Ec(X,\theta)$ which solves 
\[
\MA(\psi) =  \mathrm{\bf 1}_{X\setminus B} e^{\psi-\varphi} \MA(\varphi) + \mathrm{\bf 1}_B e^{\psi}\mu.
\]
Now, one can check  that $\varphi$ is a supersolution of the above equation while $\psi$ is a supersolution of the equation \eqref{eq: min of supersolution}. By the comparison principle and by  minimality of $\varphi$ we have that $\varphi = \psi$, finishing the proof.  
\end{proof}

We have used the following analogue of Choquet's lemma:

\begin{lem}
	\label{lem: analogue of Choquet}
	Let $\mathcal{U}$ be a family of upper semicontinuous functions on $X$. Then there exists a  countable subfamily $\mathcal{N}\subset \mathcal{U}$ such that $\inf_{u\in \mathcal{U}} u=\inf_{u\in \mathcal{N}} u$. 
\end{lem}

\begin{proof}
Replacing each function $u \in \mathcal{U}$ by $u/(1+|u|)$, we are reduced to the case when the family
$\mathcal{U}$ is uniformly bounded.

Fix a distance $d$ on $X$ (e.g. induced by a Riemann metric on $X$). For an upper semicontinuous function $v$ on $X$ we consider the sup-convolution $\Phi(v,j)$ defined by 
\[
\Phi(v,j)(x) := \sup\{v(y) -jd(x,y)\ ; \ y\in X\}, j\in \mathbb{N}.  
\]
Then for each $j\in \mathbb{N}$, $\Phi(v,j)$ is continuous on $X$ and as $j\to +\infty$ the sequence $\Phi(v,j)$ decreases pointwise to $v$.

Set $u:=\inf_{v\in \mathcal{U}} v$. By Choquet's lemma (see e.g. \cite[Lemma 4.31]{GZbook}), for each $j\in \mathbb{N}$ there exists a sequence $(\varphi_k^j)_{k\in \mathbb{N}}\subset \mathcal{U}$ such that 
\[
(\inf_k \Phi(\varphi_k^j,j))_{\star} = (\inf_{v\in \mathcal{U}} \Phi(v,j))_{\star}.
\] 
	
Set $\varphi:= \inf_{j,k} \varphi_{k}^j$. The lemma is now reduced to showing that  $\varphi=u$. Indeed, it is obvious that $\varphi\geq u$. For each $j\in \mathbb{N}$ and any $v\in \mathcal{U}$, we have 
\[
(\Phi(\varphi, j))_{\star} \leq  (\inf_k \Phi(\varphi_k^j,j))_{\star} \leq \Phi(v,j). 
\]
We observe also that $(\Phi(\varphi,j))_{\star}=\Phi(\varphi,j)$ for any $j$.  
The function $\Phi(\varphi,j)$ is continuous on $X$. We thus have $\Phi(\varphi,j)\leq \Phi(v,j)$ for all $v\in \mathcal{U}$. Note also that $\varphi, u$ are upper semicontinuous on $X$. Letting $j\to +\infty$ we get $\varphi\leq u$ completing the proof.  
\end{proof}

At the end of the proof of \ref{thmC} we could also have used a local Dirichlet problem and do a gluing process (balayage technique). We prove in the following that this process works well for measures with finite masses:

\begin{lem} \label{lem:gluing}
Let $u$ be a  $\theta$-psh function with minimal singularities  such that $\MA(u)\leq e^u \nu$, where $\nu$ is a non-pluripolar positive Radon measure on $X$.  
Let $B$ be a small ball in the ample locus of $\theta$. Let  $v$ be a bounded $\theta$-psh function in $B$ such that $v\leq u$ in $B$, $\lim_{z\to \partial B}(v-u)\geq 0$,  and    $\MA(v) \leq e^v \nu$ in $B$.  Set
 $$
 \p(x)=\left\{
 \begin{array}{ccl}
 u(x) & \text{ if } & x \in X \setminus B \\
 v(x) & \text{ if } & x \in   B.
 \end{array}
 \right.
 $$
Then $P_{\theta}(\p)$ is a (pluripotential) supersolution of the equation $\MA(\varphi)=e^{\varphi}\nu$ on $X$.
\end{lem}

\begin{proof} The proof uses the machinery we have developed so far. It consists in showing that $P_{\theta,\nu} (\psi)$ is a pluripotential supersolution, which moreover coincides with $P_{\theta} (\psi)$, by using  the Berman approximation process.

For each $\beta>1$ let $\varphi_{\beta}$ be the unique function in $\mathcal{E}^1(X,\theta)$ such that (see \cite{BBGZ13}) 
\[
\MA(\varphi_{\beta}) = e^{\beta (\varphi_{\beta}-\psi)} e^{\varphi_{\beta}} \nu
=e^{(\beta+1) (\varphi_{\beta}-\psi)} e^{\p} \nu.
\]
Since $\psi \leq u$ and $\MA(u)\leq e^{u}\mu$ one can check that $u$ is a supersolution of the above equation. It follows from the comparison principle (see Proposition \ref{prop: comparison principle exponential}) that $\varphi_{\beta}\leq u$, for all $\beta>0$. 

We claim that $\varphi_{\beta}\leq v$ in $B$, for all $\beta>0$. This could follow from the local comparison principle, if we knew that $\varphi_\beta$ belongs by restriction to a local finite energy class in $B$. Since the definition and properties of $\varphi_\beta$ are global in nature, we need to make a technical detour.

Fix $\beta>0$ and let $g$ be a local potential of $\theta$ in a neighborhood of $B$ (i.e. $dd^c g =\theta$). Fix $\rho$ a negative strictly psh  function in $B$.  Set 
$$
\phi_{\beta,j}:= g+ \max(\varphi_{\beta}, V_{\theta}-j), \, j\in \mathbb{N}
$$
 and note that $(dd^c \phi_{\beta,j})^n$ converges to $(\theta +dd^c \varphi_{\beta})^n$ as $j\to +\infty$ in the strong sense of Borel measures on $B$. Since $v$ is bounded and $v=u\geq \varphi_{\beta}$ on $\partial B$ it follows that, for $j$ big enough, $\phi_{\beta,j} \leq v$ on  $\partial B$. Fix $\vep>0$ and set 
\[
U_{\beta,\vep,j}:= B\cap \{g+v < \vep \rho+\phi_{\beta,j}\}, \ U_{\beta,\vep}:= B\cap \{v < \vep \rho + \varphi_{\beta}\}.
\]
Observe that $U_{\beta,\vep}\subset U_{\beta,\vep,j} \Subset B$ and $ \{v\leq \vep \rho + \varphi_{\beta}\}\subset \{v<\varphi_{\beta}\}$. The comparison principle for bounded psh functions \cite[Theorem 4.1]{BT82} yields
\begin{eqnarray*}
	\int_{U_{\beta,\vep}} [\vep^n (dd^c \rho)^n + (dd^c \phi_{\beta,j})^n] &\leq &   \int_{U_{\beta,\vep,j}} (dd^c (\vep \rho + \phi_{\beta,j}))^n\\
	& \leq & \int_{U_{\beta,\vep,j}} (dd^c (g+v))^n.
\end{eqnarray*}
Letting $j\to +\infty$ and using that $v$ is a supersolution we obtain 
\[
\int_{U_{\beta,\vep}} [\vep^n (dd^c \rho)^n + e^{\beta(\varphi_{\beta}-v)}e^{\varphi_{\beta}} d\nu] \leq \int_{\{v<\varphi_{\beta}\}} e^{v} d\nu. 
\] 
In $U_{\beta,\vep}$ we have $\varphi_{\beta} >v$. Thus letting $\vep\to 0$ in the above inequality we obtain $\nu(v<\varphi_{\beta})=0$. 
Since $\rho$ is strictly psh we conclude that the set $U_{\beta,\vep}$, and hence also the set $\{v<\varphi_{\beta}\}$, has Lebesgue measure zero. It is thus empty, proving the claim.

 Thus  $\varphi_{\beta}\leq \psi$ on $X$ for all $\beta>0$. It follows that $\MA(\varphi_{\beta})\leq e^{\varphi_{\beta}} \nu$. As $\beta\to +\infty$ \ref{thmB} shows that $\varphi_{\beta}$ converges in energy to $P_{\theta,\nu}(\psi)$ which satisfies $P_{\theta,\nu}(\psi)\leq \psi$ quasi everywhere on $X$. Thus $P_{\theta,\nu}(\psi) \leq P_\theta (\psi)$.  Since the inequality  $ P_\theta (\psi) \leq P_{\theta,\nu}(\psi)$ is  always satisfied for a non pluripolar measure, we obtain the equality $P_{\theta,\nu}(\psi) = P_\theta (\psi)$.   By continuity of the Monge-Amp\`ere measure along convergence in energy, this eventually shows that $\MA(P(\psi)) \leq e^{P(\psi)}\nu$. 
\end{proof}

 \section{Further applications} \label{sec:applications}

\subsection{Controlling the mass of viscosity super-solutions}
Let $\theta$ be a smooth closed $(1,1)$-form such that $[\theta]$ is big. 
\begin{pro}
	Let $f:X \rightarrow \R^+$ be a continuous function. There exists a viscosity super-solution of 
	\[
	(\theta +dd^c u)^n = f\omega^n
	\]
	if and only if $\int_X f\omega^n \geq \vol(\theta)$.
\end{pro}

\begin{proof}
	Assume that $\int_X f\omega^n \geq \vol(\theta)$. Let $\varphi\in \psh(X,\theta)$ be the unique function with minimal singularities normalized by $\sup_X \varphi=0$ such that $\MA(\varphi)=cf\omega^n$ where $c>0$ is a normalization constant. It follows from \cite{EGZ11} that $\varphi$ is also a viscosity solution. Since $c\leq 1$ the result follows. 
	
	Conversely, assume that $u$ is a viscosity supersolution.  It follows from Theorem  \ref{thm: supersolution general right hand side} that  $P_{\theta} (u)$ is also a pluripotential supersolution, hence the inequality follows.
\end{proof}

The connection between pluripotential and viscosity supersolutions of complex Monge-Amp\`ere equations
allows us to derive the following interesting inequality:

 \begin{pro}\label{prop: weak BD12}
 	Assume that $\theta$ is a closed smooth $(1,1)$-form such that $[\theta]$ is big. 
 	Then the envelope $V_{\theta}$ satisfies
 	\[
 	\MA(V_{\theta}) \leq {\bf 1}_{\{V_{\theta}=0\}} \theta^n. 
 	\]
 \end{pro}
 
This is a particular case of an important result of Berman and Demailly \cite{BD12}, which uses strong regularity
information on the function $V_{\theta}$.
We provide a proof of independent interest. A slightly different proof has recently been given in \cite{Ber13},  \cite{DDL16} using the viscosity theory developed in \cite{EGZ11}.

\begin{proof}
	The function $0$ is a viscosity super-solution of the equation 
	\[
	(\theta +dd^c u)^n = e^{u} \theta_+^n,
	\]
	where $\theta_+$ is defined pointwise to be $\theta$ if $\theta\geq 0$ and zero otherwise. It follows from Theorem \ref{thm: visco super vs pluri compact} that $V_{\theta}$ is a pluripotential super-solution of the same equation, thus $\MA(V_{\theta}) \leq e^{V_{\theta}}\theta_+^n$ in the pluripotential sense. As $\MA(V_{\theta})$ is concentrated on the contact set $\{V_{\theta}=0\}$, the conclusion follows. 
\end{proof}

\subsection{Examples of viscosity supersolutions}

As the concept of viscosity super-solutions to complex Monge-Amp\`ere equations is relatively new and still a bit mysterious,
it is probably helpful to discuss in some details a few elementary examples.

\begin{pro}\label{prop: two envelopes coincide}
Assume that $u$ is a bounded viscosity supersolution of 
\[
	(\theta +dd^c u)^n = C\omega^n,
\]
where $C>0$ is a constant. Then $P_{\theta,dV}(u) =P_{\theta}(u)$. In other words, if a $\theta$-psh function $\varphi$ satisfies $\varphi\leq u$ almost everywhere with respect to Lebesgue measure then the inequality holds quasi everywhere. 
\end{pro}

\begin{proof}
	We use the convergence method developed in Section \ref{sect: approximation}. For each $\beta>1$ let $\varphi_{\beta}$ be the unique $\theta$-psh function with minimal singularities such that 
	\begin{equation}
	\label{eq: viscosity super application}
	\MA(\varphi_{\beta}) =e^{\beta(\varphi_{\beta}-u)}\omega^n. 
	\end{equation}
	
We claim that $\varphi_{\beta}\leq u + \frac{\log C}{\beta}$ in the ample locus of $\theta$, for any $\beta>1$. Indeed, observe that $u+\log(C)/\beta$ is a viscosity supersolution of \eqref{eq: viscosity super application}. Using an approximation argument and the viscosity comparison principle as in the proof of Theorem \ref{thm: visco super vs pluri compact} we can show that $\varphi_{\beta} \leq u+  \frac{\log C}{\beta}$ in $\Amp(\theta)$. But by Theorem \ref{thm Berman generalized convergence} we know that $\varphi_{\beta}$ converges in energy to the modified envelope $P_{\theta,dV}(u)$.   It thus follows that $P_{\theta,dV}(u)=P_{\theta}(u)$. 
\end{proof}

\begin{pro} For $n=1$,
the viscosity supersolutions  
$$
(dd^c v)_+ \leq dd^c |z|^2
$$
are precisely the functions $v$ such that $v-|z|^2$ is super-harmonic.
\end{pro}

\begin{proof}
Let $q$ be a $\mathcal{C}^2$ upper test for $|z|^2-v$ at $a\in \mathbb{C}$. Then the function $|z|^2-q$ is a lower test for $v$ at $a$. It  follows that $(dd^c (|z|^2-q))_+ \leq dd^c |z|^2$ at $a$, hence $dd^c q \geq 0$ at $a$. It thus follows that $|z|^2-v$ is a viscosity subsolution, hence it is subharmonic as follows from \cite[Prop. 3.2.10, p. 147]{Hor94}. 
\end{proof}

 One could expect a similar property to hold in higher dimension: if 
 $$
(dd^c v)^n_+ \leq (dd^c |z|^2)^n
$$
in the viscosity sense, one would like to conclude that $v-|z|^2$ is $1$-concave. 
This is however not true in general :

\begin{pro}
	\label{pro: supersolution not quasi continuous}
	Let $B$ be the unit ball in $\C^2$ and consider the function $u$ defined by $u(z_1,z_2)=-1$ if $|z_1|=|z_2|$ and $u(z_1,z_2)=0$ elsewhere. Then $u$ is a viscosity supersolution of the Monge-Amp\`ere equation
	\[
	(dd^c u)^n =0. 
	\]
\end{pro}

\begin{proof}
	We set $D:=\{(z_1,z_2)\in B\, :\, |z_1|=|z_2|\}$. If $x_0\in B\setminus D$ 
	then $u$ is smooth near $x_0$ and the result follows from \cite{EGZ11}.
	
	Assume now that $x_0=(a,a)\in D$ and $q$ is a lower test function for $u$ at $x_0$. 
The function $p(z):= q(z,z)$  is a lower test function for the constant $-1$ near $a$.
 It follows that $dd^c p$ is not positive at $a$
 hence $(dd^c q)^2_{+}(x_0)=0$. Thus $u$ is a viscosity supersolution of the above equation. 
\end{proof}

\begin{rem}
We let the reader check that the Monge-Amp\`ere envelope $P_{B}(u)$ is identically $-1$ in $B$.  Its Monge-Amp\`ere measure is thus identically $0$. This is consistent with Theorem \ref{thm:supersolloc}. 	
\end{rem}

The example in Proposition \ref{pro: supersolution not quasi continuous} indicates that  viscosity supersolutions (in a local context and without boundary constraints) are in general not  quasi-continuous. One can not construct a similar example on a compact K\"ahler manifold. 
More precsiely, we have the following:

\begin{lem} Let $(X,\omega)$ be a compact K\"ahler manifold of dimension $n$ and $\theta$ be a closed smooth semipositive $(1,1)$-form on $X$ such that $\int_X \theta^n>0$. Let $E \subset X$ be a closed subset of $X$ that has zero Lebesgue measure. If  the function $u=-{\bf 1}_{E}$  satisfies 
\[
(\theta +dd^c u)^n \leq C\omega^n
\]
in the viscosity sense on $X$, then $E$ is pluripolar. 
\end{lem}
\begin{proof} Indeed, by Proposition \ref{prop: two envelopes coincide} we have that $P_{\theta,dV}(u)=P_{\theta}(u)$. Therefore since $u = 0$ a.e. in $X$, it follows that  $P_{\theta} (u) = 0$ in  $X$, which implies that  $E$ is pluripolar since $P_{\theta}(u) = h_{E,\theta}^*$ is the relative $\theta$-plurisubharmonic extremal function of $E$ (see \cite{GZ05}).

\end{proof}

Our analysis above motivates the following:
   \begin{ques}\label{GLZ question}
	Assume that $u$ is a bounded lower semicontinuous function on $X$ such that 
\[
(\omega +dd^c u)^n \leq C\omega^n,
\]
holds in the viscosity sense for some positive constant $C$. Is $u$ quasi-continuous on $X$ ? 
\end{ques}

Understanding the regularity properties of viscosity supersolutions is an important problem.
We establish a refined semi-continuity property:

\begin{pro}
Assume that $v$ is a bounded viscosity supersolution of $-(\theta+dd^c v)^n + e^v CdV=0$. Then 
$$
v(a)=\liminf_{x \rightarrow a, x \neq a} v(x)
$$
for all $a \in X$.
\end{pro}

\begin{proof}
Assume that it were not the case. Then we can find $a\in X$ and $\varepsilon>0$ such that 
\[
v(a) + \varepsilon \leq \liminf_{x\rightarrow a, x\neq a} v(x), 
\]
Thus, we can find a small ball $B(a,r)$ in a local coordinate chart around $a$ such that 
$v(a)+\varepsilon \leq \inf_{x\in B(a,r), x\neq a} v(x)$. Now, for any $A>0$ the function $q_A(z):=A|z-a|^2+v(a)$ is a smooth sub test of $v$ at $a$. If $A$ is large enough then $(\theta+dd^c q_A)$ is positive definite at $a$ and  the inequality $(\theta +dd^c q_A)_+^n \leq e^{q_A} CdV$ does not hold at $a$, which is a contradiction. 
\end{proof}

\bibliographystyle{amsbook}

\end{document}